\documentclass{amsart}

 \pdfoutput=1

\usepackage[numbers,authoryear]{natbib}
\usepackage{amsmath}
\usepackage{amsfonts}
\usepackage{amsthm}
\usepackage{amssymb}
\usepackage{txfonts}
\usepackage{graphicx}
\usepackage{hyperref}
\usepackage[english]{babel}
\usepackage{booktabs}
\usepackage[authoryear]{natbib}

\theoremstyle{plain}
\newtheorem{theorem}{Theorem}[section]

\newtheorem{lemma}[theorem]{Lemma}

\theoremstyle{definition}

\theoremstyle{remark}
\newtheorem{remark}[theorem]{Remark}

\newtheorem{assumption}[theorem]{Assumption}

\newcommand{\f}[2]{\frac{#1}{#2}}

\numberwithin{equation}{section}

\newcommand{\dt}{\Delta t}
\newcommand{\dW}{\Delta W}

\newcommand{\abs}[1]{\left\vert #1 \right\vert}
\newcommand{\norm}[1]{\left\lVert #1 \right\rVert}
\newcommand{\ttheta}{\widetilde{\theta}}
\newcommand{\R}{\mathbb{R}}
\newcommand{\E}{\mathbb{E}}
\newcommand{\N}{\mathbb{N}}
\newcommand{\set}[1]{\left\{#1\right\}}
\newcommand{\floor}[1]{\left\lfloor#1\right\rfloor}

\begin{document}

\title[Forward-reverse EM algorithm for Markov chains]{Forward-reverse EM
  algorithm for Markov chains: convergence and numerical analysis}

\thanks{Partially supported by the DFG Research Center \textsc{Matheon}
  ``Mathematics for Key Technologies'' in Berlin and the "Finance and Sustainable Growth" Laboratory of the Louis Bachelier LABEX in Paris.\\
  C.B.~is grateful to Pedro Villanova for enlightening discussions regarding
  the implementation of the algorithm leading to substantial improvements in
  Theorem~\ref{thr:complexity-fr}.
}

\author{Christian Bayer}
\address{Weierstrass Institute\\Mohrenstr. 39\\10117 Berlin\\Germany}
\email{Christian.Bayer@wias-berlin.de}

\author{Hilmar Mai}
\address{Centre de Recherche en Economie et Statistique\\ENSAE-Paris Tech\\92245 Paris\\France}
\email{hilmar.mai@ensae.fr}

\author{John Schoenmakers}
\address{Weierstrass Institute\\Mohrenstr. 39\\10117 Berlin\\Germany}
\email{John.Schoenmakers@wias-berlin.de}

\date{\today}

\begin{abstract}
  We develop a forward-reverse EM (FREM) algorithm for estimating parameters that determine the
  dynamics of a discrete time Markov chain evolving through a certain measurable state space.
  As a key tool for the construction of the FREM method we develop forward-reverse representations
  for Markov chains conditioned on a certain terminal state. These representations may be considered
  as an extension of the earlier work \cite{BS} on conditional diffusions. We proof almost sure convergence of our algorithm for a Markov chain model with curved exponential family structure. On the numerical side we give a complexity analysis of the forward-reverse algorithm by deriving its expected cost. Two application examples are discuss to demonstrate the scope of possible applications ranging from models based on continuous time processes to discrete time Markov chain models.
\end{abstract}



\maketitle

\section{Introduction}

The EM algorithm going back to the seminal paper \cite{dempster} is a very general method for iterative computation of maximum likelihood estimates in the setting of incomplete data. The algorithm consists of an expectation step (E-step) followed by a maximization step (M-step) which led to the name EM algorithm.

Due to its general applicability and relative simplicity it has nowadays found its way into a great number of applications. These include maximum likelihood estimates of hidden Markov models in \cite{macdonald}, non-linear time series models in \cite{chan} and full information item factor models in \cite{meng} to give just a very limited selection.

Despite the simplicity of the basic idea of the algorithm its implementation in more complex models can be rather challenging. The global maximization of the likelihood in the M-step has recently been addressed successfully (see e.g. \cite{meng1993} and \cite{liu}). On the other hand, when the expectation of the complete likelihood is not known in closed form only partial solutions have been given yet. One approach developed in \cite{wei} uses Monte Carlo approximations of the unknown expectation and was therefore named Monte Carlo EM (MCEM) algorithm. As an alternative procedure the stochastic approximation EM algorithm was suggested in \cite{delyon}.

In this paper we take a completely different route by using a forward-reverse algorithm (cf. \cite{BS}) to approximate the conditional expectation of the complete data likelihood. In this respect we extend the idea from \cite{BS} to a Markov chain setting, which is considered an interesting contribution on its own. Indeed, Markov chains are more general in a sense, since any
diffusion monitored at discrete times yields canonically a Markov chain, but not every chain can be embedded (straightforwardly) into some continuous time diffusion the other way around.

The central issue  is the identification of a parametric Markov
chain model $(X_{n},$ $n=0,1,\ldots)$ based on incomplete data, i.e. realizations of the
model, given on a typically course grid of time points, let us say
$n_{1},n_{2},\ldots n_{N}.$ Let us assume that the chain runs through
$\mathbb{R}^{d}$ and that the transition densities $p_{n,m}^{\theta}(x,y),$
$n\geq m,$ of the chain exist (with $p_{n,n}^{\theta}(x,y):=\delta_{x}(y)$),
where the unknown parameter $\theta$ has to be determined. The log-likelihood function based on the incomplete observations $(X_{n_1},\ldots, X_{n_N})$ is then given by
\begin{equation}\label{lh}
 l(\theta,x_{n_1},\ldots,x_{n_N})=\sum_{i=0}^{N-1} \ln p_{n_{i},n_{i+1}}^{\theta
}(x_{n_{i}},x_{n_{i+1}}),
\end{equation}
with $X_{n_{0}}=x_{0}$ being the
initial state of the chain. Then the standard
method of maximum likelihood estimation would suggest to evaluate%
\begin{equation}
\underset{\theta}{\arg\max}\; l(\theta,X)=\underset{\theta}{\arg\max} \sum_{i=0}^{N-1}\ln p_{n_{i},n_{i+1}}^{\theta
}(X_{n_{i}},X_{n_{i+1}}).\label{am}%
\end{equation}
The problem in this approach lies in the fact that usually only the one-step transition
densities $p_{n,n+1}^{\theta} (x,y)$ are explicitly known, while any multi-step
density $p_{n,m}^{\theta}(x,y)$ for $m>n$ can be expressed as an $m-n-1$ fold
integral of one-step densities. In particular for larger $m-n,$ these multiple
integrals are numerically intractable however. 

In the EM approach, we therefore consider the
alternative problem%
\begin{equation}
\underset{\theta}{\arg\max}\sum_{i=0}^{N-1}\sum_{j=n_{i}}^{n_{i+1}%
-1}\mathbb{E}\ln p_{j,j+1}^{\theta}(X_{j},X_{j+1}),\label{am1}%
\end{equation}
in terms of the \textquotedblleft missing data\textquotedblright\ $X_{n_{i}%
+1},...,X_{n_{i+1}-1},$ $i=0,...,N-1.$ As such, between two such consecutive
time points, $n_{i}$ and $n_{i+1}$ say, the chain may be considered as a
bridge process starting in realization $X_{n_{i}}$ and ending up in
realization $X_{n_{i+1}}$ (under the unknown parameter $\theta$ though), and
so each term in (\ref{am1}) may be considered as an expected functional of the
\textquotedblleft bridged\textquotedblright\ Markov chain starting at time
$n_{i}$ in (data point) $X_{n_{i}},$ conditional on reaching (data point)
$X_{n_{i+1}}$ at time $n_{i+1}.$ 

We will therefore develop firstly an
algorithm for estimating the terms in (\ref{am1}) for a given parameter
$\theta.$ This algorithm will be of forward-reverse type in the spirit of the
one in \cite{BS} developed for diffusion bridges. It should be noted here
that in the last years the problem of simulating diffusion bridges has
attracted much attention. Without pretending to be complete, see for example, \cite{BlS,DH,MT,S,SVW,SMZ}.

Having the forward-reverse algorithm at
hand, we may construct an approximate solution to (\ref{am1}) in a sequential
way by the EM algorithm: Once a generic approximation $\theta_m$ is constructed after $m$ steps, one estimates
\[
\theta_{m+1}:=\underset{\theta}{\arg\max}\sum_{i=0}^{N-1}%
\sum_{j=n_{i}}^{n_{i+1}-1}\widehat{\mathbb{E}}\ln p_{j,j+1}^{\theta}%
(X_{j}^{\theta_m},X_{j+1}^{\theta_m}),
\]
where $X^{\theta_m}$ denotes the Markov bridge process under the
transition law due to parameter $\theta_m$ and each term
\[
\widehat{\mathbb{E}}\ln p_{j,j+1}^{\theta}(X_{j}^{\theta_m%
},X_{j+1}^{\theta_m})
\]
represents a forward-reverse approximation of%
\[
\mathbb{E}\ln p_{j,j+1}^{\theta}(X_{j}^{\theta_m},X_{j+1}%
^{\theta_m})
\]
as a (known) function of $\theta.$

Convergence properties of approximate EM algorithms have drawn considerable
recent attention in the literature mainly driven by it's success in earlier
intractable estimation problems. An overview of existing convergence results
for the MCEM algorithm can be fund in \cite{neath}. Starting from a
convergence result for the forward-reverse representation for Markov chains we
prove almost sure convergence of the FREM sequence in the setting of curved
exponential families based on techniques developed in \cite{BS} and
\cite{fort}. Essentially the only ingredient from the Markov chain model for
this convergence to hold are exponential tails of the transition densities and
their derivatives that are straightforward to check in many examples.

Since computational complexity is always an issue in estimation techniques
that involve a simulation step, we also include a complexity analysis for the
forward-reverse algorithm. We show that the algorithm achieves an expected
cost of the order $O(N \log N )$ for sample size of $N$ forward-reverse
trajectories.

We mention a recent application paper by Bayer, Moraes, Tempone, and Vilanova
\cite{BMTV}, which focuses on practical and implementation issues of the
forward-reverse EM algorithm in the setting of Stochastic Reaction Networks
(SRNs), i.e., of continuous time Markov chains with discrete, but generally
infinite state space. 

The structure of the paper is as follows. In Section \ref{sec2} we
recapitulate and adapt the concept of reversed Markov chains, initially
developed in \cite{MSS2} using the ideas in \cite{MSS1} on reversed
diffusions. A general stochastic representation --- involving standard
(unconditional) expectations only --- for expected functionals of
conditional Markov chains is constructed in Section \ref{sec3}. This
representation allows for a forward reverse EM algorithm that is introduced
and analyzed in Section \ref{sec4}. In Section \ref{sec:conv} we proof almost
sure convergence of the forward-reverse EM algorithm in the setting of curved
exponential families. Implementation and Complexity of the FREM algorithm are
addressed in Section \ref{sec:impl-forrev}. The paper is concluded with two
application examples in Section \ref{sec:appl} that demonstrate the wide scope
of our method.

\section{Recap of forward and reverse representations for Markov chains}\label{sec2}

Consider a discrete-time Markov process $(X_{n},\mathcal{F}_{n}%
),\ n=0,1,2,...,$ on a probability space $(\Omega,\mathcal{F},$ $\mathbb{P})$
with phase space $(S,\mathcal{S})$, henceforth called Markov chain. In general
we assume that $S$ is locally compact and that $\mathcal{S}$ is the Borel
$\sigma$-algebra on $S.$ For example, $S=\mathbb{R}^{d}$ or a proper subset
of $\mathbb{R}^{d}.$ Let $P_{n},\ n\geq0,$ denote the one-step transition
probabilities defined by
\begin{equation}
P_{n}(x,B):=\mathbb{P}(X_{n+1}\in B\ |\ X_{n}=x),\ n=0,1,2,...,\ \ x\in
S,\ B\in\mathcal{S}. \label{Ch0}%
\end{equation}
In the case of an autonomous Markov chain all the one-step transition
probabilities coincide and are equal to $P:=P_{0}=P_{1}=\cdot\cdot\cdot.$

Let $X_{m}^{n,x},$\ $m\geq n,$ be a trajectory of the Markov chain which is at
step $n$ in the point $x,$ i.e., $X_{n}^{n,x}=x.$ The multi-step transition
probabilities $P_{n,m}$ are then defined by%
\[
P_{n,m}(x,B):=\mathbb{P}(X_{m}^{n,x}\in B),\ \ x\in S,\ \ B\in\mathcal{S}%
,\ \ m\geq n.
\]
Due to these definitions, $P_{n,n}(x,B)=\delta_{x}(B)=1_{B}(x)$ (Dirac
measure)$,$ $P_{n}=P_{n,n+1},$ and the Chapman - Kolmogorov equation has the
following form:
\begin{equation}
P_{n,m}(x,B)=\int P_{n,k}(x,dy)P_{k,m}(y,B),\ \ \ \ x\in S,\ B\in
\mathcal{S},\ n\leq k\leq m. \label{Ch1}%
\end{equation}
Let us fix $N>0$ and consider for $0\leq n\leq N$ the function
\begin{equation}
u_{n}(x):=\int P_{n,N}(x,dy)f(y)=\E\ f(X_{N}^{n,x}), \label{Ch2}%
\end{equation}
where $f$ is $\mathcal{S}$-measurable and such that the mathematical
expectation in (\ref{Ch2}) exists; for example, $f$ is bounded. By the Markov
property we have for $0\leq n<N:$%
\begin{align*}
u_{n}(x)  &  =\E\ f(X_{N}^{n,x})=\E\ f(X_{N}^{n+1,X_{n+1}^{n,x}})\\
&  =\E\ \E^{\mathcal{F}_{n+1}}f(X_{N}^{n+1,X_{n+1}^{n,x}})=\E\ \E^{X_{n+1}^{n,x}%
}f(X_{N}^{n+1,X_{n+1}^{n,x}})\\
&  =\E\ u_{n+1}(X_{n+1}^{n,x})=\int u_{n+1}(y)P_{n}(x,dy).
\end{align*}
Thus, $u_{n}(x)$ satisfies the following discrete integral Cauchy problem
\begin{align}
u_{n}(x)  &  =\int u_{n+1}(y)P_{n}(x,dy),\ n<N,\label{Ch3}\\
u_{N}(x)  &  =f(x), \label{Ch4}%
\end{align}
and (\ref{Ch2}) is a forward probabilistic representation of its solution. In
fact, the probabilistic representation (\ref{Ch2}) can be used for simulating
the solution of (\ref{Ch3})-(\ref{Ch4}) by Monte Carlo. For our purpose,
reverse probabilistic representations we need a somewhat more general version
of the above result.

\begin{theorem}[cf.~\cite{MSS2}]
\label{Theorem1} Let $P_{n}$ be the one-step transition density of a Markov
chain $X$ as in (\ref{Ch0}) and let the function $f:$ $S\rightarrow\mathbb{R}$
be measurable and bounded. Let further $\varphi_{n}:$ $S\times S\rightarrow
\mathbb{R}$ be a measurable and bounded functions for $n=0,1,2,...$ Then, the
solution of the problem%
\begin{align}
w_{n}(x)  &  =\int w_{n+1}(z)\varphi_{n}(x,z)P_{n}%
(x,dz),\ \ \ n<N,\label{SysT1}\\
w_{N}(x)  &  =f(x) \label{SysT1a}%
\end{align}
\ has the following probabilistic representation:%
\begin{equation}
w_{n}(x)=\E\left[  f(X_{N}^{n,x})\mathcal{X}_{N}^{n,x,1}\right]  , \label{PU}%
\end{equation}
where $(X,\mathcal{X})$ is an extended Markov chain in which $\mathcal{X}$ is
governed by the equations%
\begin{equation}
\mathcal{X}_{k+1}^{n,x,\gamma}=\mathcal{X}_{k}^{n,x,\gamma}\varphi_{k}%
(X_{k}^{n,x},X_{k+1}^{n,x}),\qquad\mathcal{X}_{n}^{n,x,\gamma}=\gamma
,\nonumber
\end{equation}
where $n\leq k<N.$
\end{theorem}

\begin{proof}
Note that $\mathcal{X}_{k}^{n,x,\gamma}=\gamma\mathcal{X}_{k}^{n,x,1}.$ Thus,
for $n<N,$ (\ref{PU}) may be written as%
\begin{align*}
  w_{n}(x) &= \E\left[  f(X_{N}^{n+1,X_{n+1}^{n,x}})\mathcal{X}_{N}^{n+1,X_{n+1}%
      ^{n,x},\mathcal{X}_{n+1}^{n,x,1}}\right] \\
  &= \E\ \mathcal{X}_{n+1}^{n,x,1}\E^{(X_{n+1}^{n,x},\mathcal{X}_{n+1}^{n,x,1}%
    )}\left[  f(X_{N}^{n+1,X_{n+1}^{n,x}})\mathcal{X}_{N}^{n+1,X_{n+1}^{n,x}%
      ,1}\right] \\
  &= \E\ \left[  \mathcal{X}_{n+1}^{n,x,1}w_{n+1}(X_{n+1}^{n,x})\right] \\
  &=\E\ \left[  \varphi_{n}(x,X_{n+1}^{n,x})w_{n+1}(X_{n+1}^{n,x})\right] \\
  &=\int w_{n+1}(z)\varphi_{n}(x,z)P_{n}(x,dz),
\end{align*}
and (\ref{SysT1a}) is trivially fulfilled for $n=N.$
\end{proof}

\subsection{Reverse probabilistic representations}

We henceforth take $\left(  S,\mathcal{S}\right)  =\left(  \mathbb{R}%
^{d},\mathcal{B}(\mathbb{R}^{d})\right)  $ and assume that the transition
probabilities $P_{n,m}(x,dy)$ have densities $p_{n,m}(x,y)$ with respect to
the Lebesgue measure on $(S,\mathcal{S}).$ We note however that without any
problem one may consider more general state spaces equipped with some
reference measure, and transition probabilities absolutely continuous to with
respect to it. The representation (\ref{Ch2}) can thus be written in the form
\begin{equation}
I(f):=\E\ f(X_{N}^{n,x})=\int p_{n,N}(x,y)f(y)dy,\text{ \ \ }0\leq n\leq N.
\label{R1}%
\end{equation}
Let the initial value $\xi$ of the chain $X$ at moment $n$ be random with
density $g(x).$ Consider the functional
\begin{equation}
I(g,f)=\int\int g(x)p_{n,N}(x,y)f(y)dxdy=\E f(X_{N}^{n,\xi}). \label{R2}%
\end{equation}
Formally, by taking for $g$ a $\delta$-function we obtain (\ref{R1}) again,
and by taking $f$ to be a $\delta$-function we obtain the integral%
\begin{equation}
J(g):=\int g(x)p_{n,N}(x,y)dx. \label{R3}%
\end{equation}
We now propose suitable (reverse) probabilistic representations for $J(g),$
where $g$ is an arbitrary test function (not necessarily a density). For this
we are going to construct a class of reverse Markov chains that allow for a
probabilistic representation for the solution of (\ref{R3}).

Let us fix a number $N\in\mathbb{N}$ and consider for $0\leq m<N,$ functions
$\psi_{m}:$ $S\times S\rightarrow\mathbb{R}_{+}$ such that for each $m$ and
$y$ the function
\begin{equation}
q_{m}(y,\cdot):=\frac{p_{N-m-1}(\cdot,y)}{\psi_{m}(y,\cdot)},\text{ \ \ }0\leq
m<N, \label{varfi}%
\end{equation}
is a density on $S.$ For example, one could take $\psi_{m}$ independent of the
second argument, and then obviously%
\begin{equation}
\psi_{m}(y)=\int p_{N-m-1}(z,y)dz. \label{exa}%
\end{equation}
We now introduce a \textquotedblleft reverse\textquotedblright\ processes
$(Y_{m}^{y},\mathcal{Y}_{m}^{y})_{0\leq m\leq N}$ by the system%
\begin{gather}
\mathbb{P}(\left.  Y_{m+1}^{y}\in dz^{\prime}\right\vert \text{ }Y_{m}%
^{y}=z)=q_{m}(z,z^{\prime})dz^{\prime},\nonumber\\
\mathcal{Y}_{m+1}^{y}=\mathcal{Y}_{m}^{y}\psi_{m}(Y_{m}^{y},Y_{m+1}%
^{y}),\label{Yrev}\\
Y_{0}^{y}:=Y_{0}^{0,y}:=y,\text{ \ \ }\mathcal{Y}_{0}^{y}:=\mathcal{Y}%
_{0}^{0,y,1}:=1,\text{ \ \ }0\leq m<N,\nonumber
\end{gather}
hence $Y^{y}$ is governed by the one-step transition probabilities
$Q_{m}(z,dz^{\prime}):=$ $q_{m}(z,z^{\prime})dz^{\prime}$ (i.e. $Q_{m}$
instead of $P_{m}$).

\begin{theorem}
\label{rev_th} For any $n,$ $0\leq n\leq N,$ (\ref{R3}) has the following
probabilistic representation.%
\[
\int g(x)p_{n,N}(x,y)dx=\E\left[  g(Y_{N-n}^{y})\mathcal{Y}_{N-n}^{y}\right]
,
\]
where $g$ is an arbitrary test function (a \textquotedblright
density\textquotedblright\ $p_{m,m}$ has to be interpreted as a Dirac
distribution or $\delta$-function).
\end{theorem}

\begin{proof}
From the Chapman - Kolmogorov equation (\ref{Ch1}) we obtain straightforwardly
the Chapman-Kolmogorov equation for densities,
\begin{equation}
p_{n,m}(x,y)=\int p_{n,k}(x,z)p_{k,m}(z,y)dz,\ \ x,y\in S,\ n\leq k\leq m.
\label{R4}%
\end{equation}
Let us now fix $n,$ $n<N$ (for $n=N$ the statement is trivial) also, and
introduce the functions%
\begin{equation}
v_{k}(y):=\int g(x)p_{n,k}(x,y)dx,\ n\leq k\leq N. \label{R45}%
\end{equation}
From (\ref{R4}) we get
\begin{align}
v_{k}(y)  &  =\int v_{k-1}(z)p_{k-1}(z,y)dz,\ n<k\leq N,\label{R5}\\
v_{n}(y)  &  =g(y),\nonumber
\end{align}
where $p_{k-1}:=p_{k-1,k}$ denote the one-step densities. For $\ n<k\leq N$ we
now consider a \textquotedblleft reversed\textquotedblright\ time variable
$m=N+n-k$ and write with $\widetilde{v}_{m}(y):=v_{N+n-m}(y)$ and
(\ref{varfi}) system (\ref{R5}) in the form%

\begin{align}
\widetilde{v}_{m}(y)  &  =\int\widetilde{v}_{m+1}(z)\psi_{m-n}(y,z)q_{m-n}%
(y,z)dz,\ \ \ n\leq m<N,\label{R6_1}\\
\widetilde{v}_{N}(y)  &  =g(y).\nonumber
\end{align}
Let us write (\ref{R6_1}) in a slightly different form,%
\begin{align*}
\widetilde{v}_{m}(y)  &  =\int\widetilde{v}_{m+1}(z)\psi_{m}^{(n)}%
(y,z)q_{m}^{(n)}(y,z)dz,\ \ \ n\leq m<N,\\
\widetilde{v}_{N}(y)  &  =g(y)
\end{align*}
with $\psi_{m}^{(n)}:=\psi_{m-n}$ and $q_{m}^{(n)}:=q_{m-n}.$ Via
Theorem~\ref{Theorem1} we next obtain a probabilistic representation of the
form (\ref{PU}) for the solution of problem (\ref{R6_1}), hence (\ref{R3}) or
$J(g).$ Indeed, by taking in Theorem~\ref{Theorem1} instead of $X$ a Markov
chain $\left(  Y_{m}^{(n),y}\right)  _{n\leq m\leq N},$ where $Y^{(n),y}$ is
governed by the one-step transition probabilities $Q_{m}^{(n)}(z,dz^{\prime
}):=$ $q_{m}^{(n)}(z,z^{\prime})dz^{\prime},$ $n\leq m<N,$ with initial
condition $Y_{n}^{(n),y}=y,$ and constructing $\left(  \mathcal{Y}_{m}%
^{(n),y}\right)  _{n\leq m\leq N}$ according to%
\begin{equation}
\mathcal{Y}_{m+1}^{(n),y}=\mathcal{Y}_{m}^{(n),y}\psi_{m}^{(n)}(Y_{m}%
^{(n),y},Y_{m+1}^{(n),y}),\qquad\mathcal{Y}_{n}^{(n),y}=1,\qquad n\leq m<N,
\label{EE}%
\end{equation}
it follows by Theorem ~\ref{Theorem1} that%
\begin{equation}
J(g)=\widetilde{v}_{n}(y)=v_{N}(y)=\E\left[  g(Y_{N}^{(n),y})\mathcal{Y}%
_{N}^{(n),y}\right]  . \label{RevProb}%
\end{equation}
It remains to see that
\[
\E\left[  g(Y_{N}^{(n),y})\mathcal{Y}_{N}^{(n),y}\right]  =\E\left[
g(Y_{N-n}^{y})\mathcal{Y}_{N-n}^{y}\right]
\]
which follows from the fact that initial values and the one step transition
probabilities of the processes
\[
\left(  Y_{n+i}^{(n),y},\mathcal{Y}_{n+i}^{(n),y}\right)  _{i=0,...,N-n}\text{
\ \ and \ \ }\left(  Y_{i}^{y},\mathcal{Y}_{i}^{y}\right)  _{i=0,...,N-n}%
\]
coincide.
\end{proof}

It should be stressed that, in contrast to a corresponding theorem in
\cite{MSS2}, Theorem~\ref{rev_th} provides a family of probabilistic
representations indexed by $n=1,\ldots,N,$ that involves only one common
reverse process $Y^{y}.$ In Theorem~\ref{rev_th} $N$ was fixed but, when
different $N$ are in play, we will denote them by $Y^{y;N}.$ It turns out that
this extension of the related result in \cite{MSS2} is crucial for deriving
probabilistic representations for conditional Markov chains below (cf.
\cite{BS}).


\section{Simulation of conditional expectations via forward-reverse
representations} \label{sec3}

In this section we describe for a Markov Chain (\ref{Ch0}) an efficient
procedure for estimating the final distributions of a chain $X=(X_{n}%
)_{n=0,...,N}$ conditioned, or pinned, on a terminal state $X_{N}.$ More
specifically, for some given (unconditional) process $X$ we aim at
simulation of the functional
\begin{equation}
\mathbb{E}\left[  \left.  g(X_{m_{1}},\ldots,X_{m_{r}})\right\vert \text{
}X_{N}=y,\,X_{0}=x\right]  , \label{cp}%
\end{equation}
where $0\leq m_{1}<m_{2}<\cdot\cdot\cdot<m_{r}<N$ (hence $r<N$), $g$ is an arbitrarily given
suitable test function, and $x, y\in\mathbb{R}^{d}$ are given states. The
procedure proposed below is in fact an extension of the method
developed in \cite{BS} to discrete time Markov chains. We note that similar
techniques as in \cite{BS} also allow us to treat the more general problem
\begin{equation*}
\mathbb{E}\left[  \left.  g(X_{m_{1}},\ldots,X_{m_{r}})\right\vert \text{
}X_{N}\in A,\,X_{0}=x\right]
\end{equation*}
for suitable sets $A \subset \mathbb{R}^d$.

\subsection{Forward-reverse representations of conditional expectations}
Let us consider the problem (\ref{cp}) for fixed $x,y\in\mathbb{R}^{d}$
(i.e. $A=\{y\}$). We firstly state the following central theorem.

\begin{theorem}
\label{key} Given a grid $\mathcal{D}_{l}:=\{0\leq n^{\ast}<n_{1}<\cdot
\cdot\cdot<n_{l} =: N\},$ it holds that
\begin{multline*}
\mathbb{E}\left[  f(Y_{n_{l}-n_{0}}^{y;n_{l}},Y_{n_{l}-n_{1}}^{y;n_{l}}%
,\ldots,Y_{n_{l}-n_{l-1}}^{y;n_{l}})\mathcal{Y}_{n_{l}-n_{0}}^{y;n_{l}}\right]
\\
=\int_{\mathbb{R}^{d\times L}}f(y_{0},y_{1},\ldots,y_{l-1})\prod_{i=1}%
^{l}p_{n_{i-1},n_{i}}(y_{i-1},y_{i})dy_{i-1}%
\end{multline*}
with $y_{l}:=y$ and $n_{0}:=n^{\ast}.$
\end{theorem}

\begin{proof}
  Without loss of generality, we assume in this proof that the grid satisfies
  $n_{i} - n_{i-1} = 1$, $i=1, \ldots, l$. Indeed, extend $f:
  \mathbb{R}^{d\times l} \to \mathbb{R}$ to a function $\tilde{f}:
  \mathbb{R}^{d \times (N-n^\ast)} \to \mathbb{R}$ such that
  \begin{equation*}
   \tilde{f}\left( Y_{N-n^\ast}^{y;N}, Y_{N-n^\ast-1}^{y;N}, \ldots,
     Y_{2}^{y;N}, Y_{1}^{y;N} \right) = f\left(
     Y_{n_l-n_0}^{y;N}, Y_{n_l-n_1}^{y;N}, \ldots, Y_{n_l-n_{l-1}}^{y;N} \right).
  \end{equation*}
  Then, re-expressing the transition densities $p_{n_{i-1}, n_i}$ in terms of
  the one-step transition densities $p_i$ using Chapman-Kolmogorov, we see
  that the statement of the theorem is equivalent to
  \begin{multline}
    \label{eq:1}
    \mathbb{E}\left[ \tilde{f}\left( Y_{N-n^\ast}^{y;N},
        Y_{N-n^\ast-1}^{y;N}, \ldots, Y_{1}^{y;N} \right)
      \mathcal{Y}^{y;N}_{N-n^\ast} \right] \\=
    \int_{\mathbb{R}^{d\times(N-n^\ast)}} \tilde{f}\left( y_{n^\ast}, \ldots,
      y_{N-1} \right) \prod_{i=n^\ast+1}^{N} p_{i-1}(y_{i-1}, y_i) dy_{i-1}
  \end{multline}
  with $y_{N} \equiv y$.
  In fact, we shall prove that
  \begin{equation}
    \label{eq:2}
    \mathbb{E}\left[ f_p\left( Y^{y;N}_{p}, \ldots, Y^{y;N}_{1} \right)
      \mathcal{Y}^{y;N}_{p} \right] =
    \int f_p(y_{N-p}, \ldots, y_{N-1}) \prod_{i=N-p+1}^N p_{i-1}(y_{i-1},
    y_i) dy_{i-1}
  \end{equation}
  for any $1 \le p \le N-n^\ast$ for any (e.g., bounded measurable) function
  $f_p : \mathbb{R}^{d\times p} \to \mathbb{R}$. \eqref{eq:2} gives the
  formula from the statement of the theorem for $p=N-n^\ast$ with
  $f_{N-n^\ast}$ being the function $\tilde{f}$ from above. We
  prove~\eqref{eq:2} by induction on $p$. For $p=1$, this boils down to
  Theorem~\ref{rev_th} with $n=N-1$.

  For the step from $p-1$ to $p$, we note that by definition
  \begin{equation*}
    \mathcal{Y}_p^{y;N} = \mathcal{Y}_{p-1}^{y;N} \psi_{p-1}\left(
      Y_{p-1}^{y;N}, Y_p^{y;N} \right),
  \end{equation*}
  with $\psi_{p-1}(y, \cdot) q_{p-1}(y, \cdot) = p_{N-(p-1)-1}(\cdot, y) =
  p_{N-p}(\cdot, y)$ by~\eqref{varfi}. Hence, we have
  \begin{align*}
    \mathbb{E}\left[ f_p(Y_{p}^{y;N}, Y_{p-1}^{y;N} ,\ldots, Y_{1}^{y;N})
      \mathcal{Y}_{p}^{y;N} \right] &= \mathbb{E}\left[ f_p(Y_{p}^{y;N},
      Y_{p-1}^{y;N} ,\ldots, Y_{1}^{y;N}) \mathcal{Y}_{p-1}^{y;N}
      \psi_{p-1}(Y_{p-1}^{y;N}, Y_{p}^{y;N}) \right] \\
    &= \mathbb{E}\left[ g\left( Y_{p-1}^{y;N}, \ldots, Y_1^{y;N} \right)
      \mathcal{Y}^{y;N}_{p-1} \right],
  \end{align*}
  with
  \begin{align*}
    g(z_{p-1}, \ldots, z_1) &\equiv \mathbb{E}\left[ \left. f_p\left( Y_{p}^{y;N},
        Y_{p-1}^{y;N}, \ldots, Y_1^{y;N} \right) \psi_{p-1}(Y_{p-1}^{y;N},
      Y_{p}^{y;N}) \right| Y_{p-1}^{y;N} = z_{p-1}, \ldots, Y^{y;N}_1 = z_1 \right]\\
  &= \int f_p(z, z_{p-1}, \ldots, z_1) p_{N-p}(z, z_{p-1}) dz.
  \end{align*}
  Applying the induction hypothesis for $f_{p-1} = g$, we obtain
  \begin{align*}
    \mathbb{E}\left[ f_p(Y_{p}^{y;N}, Y_{p-1}^{y;N} ,\ldots, Y_{1}^{y;N})
      \mathcal{Y}_{p}^{y;N} \right] &= \mathbb{E}\left[ g\left( Y_{p-1}^{y;N},
        \ldots, Y_1^{y;N} \right) \mathcal{Y}^{y;N}_{p-1} \right]\\
    &= \int g(y_{N-p+1}, \ldots, y_{N-1})
    \prod_{i=N-p+2}^N p_{i-1}(y_{i-1}, y_{i}) dy_{i-1} \\
    &= \int f_p(y_{N-p}, \ldots, y_{N-1}) \prod_{i=N-p+1}^N p_{i-1}(y_{i-1},
    y_i) dy_{i-1}.\qedhere
  \end{align*}
\end{proof}

We now consider an extended integer sequence%
\[
0<m_{1}<\cdot\cdot\cdot<m_{k}=n^{\ast}=n_{0}<n_{1}<\cdot\cdot\cdot<n_{l}=N,
\]
and a kernel $K_{\epsilon}$ of the form
\[
K_{\epsilon}(u):=\epsilon^{-d}K(u/\epsilon),\quad y\in\mathbb{R}^{d},
\]
with $K$ being integrable on $\mathbb{R}^{d}$ and $\int_{\mathbb{R}^{d}%
}K(u)du=1.$ Formally $K_{\epsilon}$ converges to the delta function
$\delta_{0}$ on $\mathbb{R}^{d}$ (in distribution sense) as $\epsilon
\downarrow0.$ We then have the following stochastic representation (involving
standard expectations only) for (\ref{cp}) \ with $n_{i}=m_{k+i},$
$i=0,...,l=r-k+1.$

\begin{theorem}
\label{FRR} Let the chain $\left(  Y,\mathcal{Y}\right)  :=\left(
Y^{y;N},\mathcal{Y}^{y;N}\right)  $ be given by (\ref{Yrev}), and the modified
integer sequence $\left(  \widehat{n}_{\cdot}\right)  $ be defined by%
\begin{equation}
\widehat{n}_{i}:=n_{l}-n_{l-i},\quad i=1,\ldots,l.\label{eq:hat-grid}%
\end{equation}
It then holds%
\begin{gather}
\mathbb{E}\left[  \left.  g(X_{m_{1}},\ldots,X_{m_{r}})\right\vert \ X_{m_{0}%
}=x,\ X_{N}=y\right]  \nonumber\\
=\mathbb{E}\left[  \left.  g(X_{m_{1}},\ldots,X_{m_{k-1}},X_{n^{\ast}}%
^{0,x},X_{n_{1}},\ldots,X_{n_{l-1}})\right\vert \ X_{m_{0}}=x,\,X_{N}%
=y\right]  \nonumber\\
=\lim_{\epsilon\rightarrow0}\frac{\mathbb{E}\left[  g\left(  X_{m_{1}}%
^{m_{0},x},\ldots,X_{m_{k-1}}^{m_{0},x},X_{n^{\ast}}^{m_{0},x},Y_{\widehat
{n}_{l-1}}^{y;N},\ldots,Y_{\widehat{n}_{1}}^{y;N}\right)  K_{\epsilon}\left(
Y_{\widehat{n}_{l}}^{y;N}-X_{n^{\ast}}^{m_{0},x}\right)  \mathcal{Y}%
_{\widehat{n}_{l}}^{y;N}\right]  }{\mathbb{E}\left[  K_{\epsilon}\left(
Y_{\widehat{n}_{l}}^{y;N}-X_{n^{\ast}}^{m_{0},x}\right)  \mathcal{Y}%
_{\widehat{n}_{l}}^{y;N}\right]  }.\label{eq:for-rev-intro}%
\end{gather}

\end{theorem}

\begin{proof}
The proof is analogous to the corresponding one in \cite{BS}. As a
rough sketch, apply Theorem \ref{key} to%
\[
f(X_{m_{1}}^{0,x},\ldots,X_{n^{\ast}}^{0,x},y_{0},y_{1},\ldots,y_{l-1}%
):=g(X_{m_{1}}^{0,x},\ldots,X_{n^{\ast}}^{0,x},y_{1},\ldots,y_{l-1}%
)K_{\epsilon}(y_{0}-X_{n^{\ast}}^{0,x}),
\]
conditional on $X_{m_{1}}^{0,x},\ldots,X_{n^{\ast}}^{0,x},$ send
$\epsilon\rightarrow0,$ and divide the result by
\[
p_{0,N}(x,y)=\lim_{\epsilon\rightarrow0}\mathbb{E}\left[  K_{\epsilon}\left(
Y_{\widehat{n}_{l}}-X_{n^{\ast}}\right)  \mathcal{Y}_{\widehat{n}_{l}}\right].
\qedhere 
\]

\end{proof}

\subsection*{Forward-Reverse algorithm}\label{algo} 
Given Theorem~\ref{FRR} the corresponding forward-reverse Monte Carlo
estimator for (\ref{eq:for-rev-intro}) suggests itself: Sample i.i.d. copies
$X^{0,x,(1)},...,X^{0,x,(M)}$ of the process $X^{0,x}$ and, independently,
i.i.d. copies $\left(  Y^{y;N,(1)},\mathcal{Y}^{y;N,(\widetilde{M})}\right)
,...,\left(  Y^{y;N,(1)},\mathcal{Y}^{y;N,(\widetilde{M})}\right)  $ of the
process $\left(  Y^{y;N},\mathcal{Y}^{y;N}\right)  .$ Take for $K$ a second
order kernel, take for simplicity $M=\widetilde{M},$ and choose a bandwidth
$\epsilon_{M}\sim M^{-1/d}$ if $d\leq4,$ or $\epsilon_{M}\sim M^{-2/(4+d)}$ if
$d\geq4.$ By next replacing the expectations in the numerator and denominator
of (\ref{eq:for-rev-intro}) by their respective Monte Carlo estimates involving double sums, one
ends up with an estimator with Root-Mean-Square error $O(M^{-1/2})$ in the
case $d\leq4$ and $O(M^{-4/(4+d)})$ in the case $d>4$ (cf. \cite{BS} for
details).

\section{The forward-reverse EM algorithm}\label{sec4}

Let us now formulate the forward-reverse EM (FREM) algorithm in the setting of the missing data problem. Suppose that the parameter $\theta \in \Theta \subset \mathbb{R}^s$ and that the Markov chain $X=(X_n,n\in \mathbb{N})$ has state space $\mathbb{R}^d$. Assuming that the transition densities $p_{n,k}$ of $X$ exist for $n,k \in \mathbb{N}$ the full-data $\log$-likelihood then reads
\begin{equation}
 l_c(\theta,x)=\sum_{i=0}^{n_N-1}\log p_{i,i+1}^{\theta
}(x_{i},x_{i+1}), \quad x \in \mathbb{R}^{n_N+1}.
\end{equation}
In the missing data problem only partial observations $X_{n_0},\ldots, X_{n_N}$ are available for $0=n_0 <n_1< \ldots < n_N$ with log-likelihood function $l$ given in \eqref{lh} that is intractable in most cases. Instead, the maximization of $l$ in $\theta$ has to be replaced by a two step iterative procedure, the EM algorithm.
\begin{description}
 \item[E-step] In the $m$-th step evaluate the conditional expectation of the complete data $\log$-likelihood
 \[
  Q(\theta, \theta_m,x):= \mathbb{E}_{\theta_m} [l_c(\theta,X_0,\ldots,X_{n_N}) | X_{n_0}= x_{n_0},\dots,X_{n_N}=x_{n_N}], \quad x \in \mathbb{R}^{N+1}.
 \]
  \item[M-step] Update the parameter by
  \[
   \theta_{m+1} = \arg \max_\theta Q(\theta,\theta_m,X).
  \]
\end{description}
Since in many Markov chain models the E-step is intractable in this form, we propose a forward-reverse approximation for the expectation of the transition densities evaluated at the observations.
\begin{description}
 \item[FR E-step] Evaluate
 \begin{equation}
  Q_m(\theta, \theta_m,X):= \mathbb{E}_{\theta_m}^{FR} [l_c (\theta,X) | X_{n_i}, i = 0, \ldots,N],
 \end{equation}
where $\mathbb{E}_{\theta_m}^{FR}$ denotes a forward-reverse approximation of the conditional expectation under the parameter $\theta_m$.
\end{description}
After this FR E-step is computed the M-step remains unchanged. This FREM algorithm gives a random sequence $(\theta_n)_{n \geq 0}$ that under certain conditions given in the next section converges to stationary points of the likelihood function. To assure a.s. boundedness of this sequence we apply a stabilization technique introduce in \cite{chen1988}.

\subsubsection*{The stable FREM algorithm}

Let $K_m \subset \Theta$ for $m \in \mathbb{N}$ be a sequence of compact sets such that
\begin{equation}\label{def:Kn}
 K_m \subsetneq K_{m+1} \quad \text{and} \quad \Theta = \bigcup_{m \in \mathbb{N}} K_m
\end{equation}
for all $m \in \mathbb{N}$. We define the stable FREM algorithm by checking if $\theta_m$ after the $m$-th maximization step lies in $K_m$ and reseting the algorithm otherwise. Choose a starting value $\theta_0 \in K_0$ and let $p_n$ for $n \in \mathbb{N}$, $p_0 :=0$, count the number of resets.
\begin{description}
 \item[stable M-step]
 \begin{align}
  \theta_{m+1} &= \arg \max_\theta Q_m (\theta,\theta_m,X) \text{ and } p_{n+1} = p_n,& \quad \text{ if } \arg \max_\theta Q_m (\theta,\theta_m,X) \in K_m,\\
  \theta_{m+1} &= \theta_0 \text{ and } p_{n+1} = p_n+ 1,& \quad \text{ if } \arg \max_\theta Q_m (\theta,\theta_m,X) \notin K_m.
 \end{align}
\end{description}
We will show in the next section that under weak assumption the number of resets $p_n$ stays a.s. finite. Our stable FREM algorithm consists now of iteratively repeating the FR E-step and the stable M-step.

\section{Almost sure convergence of the FREM algorithm} \label{sec:conv}

In this section we prove almost sure convergence of the stable FREM algorithm under the assumption that the complete data likelihood is from a curved exponential family. Our proof is mainly based on results from \cite{BS}, \cite{fort} and the classical framework for the EM algorithm introduced in \cite{dempster} and \cite{lange}.

\subsection{Model setting}
Suppose that $\phi:\Theta \to \mathbb{R}$, $\psi: \Theta \to \mathbb{R}^q$ and $S:\mathbb{R}^{(N+1)d} \to \mathbb{R}^q$ are continuous functions. We make the structural assumption that the full data log-likelihood is of the form
\begin{equation}
 l(\theta,x_0, \ldots, x_N) = \phi(\theta)+ \langle S(x_0,\ldots,x_N),\psi(\theta)\rangle,
 \end{equation}
i.e. $l$ is from a curved exponential family. In order to proof convergence we need the following properties to be fulfilled that naturally hold in many popular models. In Section \ref{sec:impl-forrev} we give several practical  examples that fall into this setting.
\begin{assumption}\label{ass1}
 \begin{enumerate}
  \item There exists a continuous function $\bar \theta: \mathbb{R}^q \to \Theta$ such that $l(\bar \theta(s),s) = \sup_{\theta \in \Theta} l(\theta,s)$ for all $s \in \mathbb{R}^{(N+1)d}$.
  \item The incomplete data likelihood $L$ is continuous in $\theta$, and the level sets $\{ \theta \in \Theta | L(
  \theta,x) \geq C \}$ are compact for any $C > 0$ and all $x$.
  \item The conditional expectation $\mathbb{E}_\theta [S(X_0, \ldots, X_{n_N})| X_{n_0}= x_{n_0},\dots,X_{n_N}=x_{n_N}]$ exists for all $(x_{n_0},\ldots, x_{n_N}) \in \mathbb{R}^{N+1}$ and $\theta \in \Theta$ and is continuous on $\Theta$.
 \end{enumerate}
\end{assumption}

To simplify our notation we will neglect in the following the dependence of
$l$ on $s$. Under these assumption we can separate the E- and M-step. In order to do so we define
\[
 g(\theta):= \mathbb{E}_\theta [S(X_0, \ldots, X_{n_N})| X_{n_0}= x_{n_0},\dots,X_{n_N}=x_{n_N}].
\]
An iteration of the EM algorithm can now be written as $\theta_{m+1} = \bar{\theta} \circ g(\theta_m) =:T(\theta_m)$. Let us denote by $\Gamma$ the set of stationary points of the EM algorithm, i.e.
\[
 \Gamma = \{\theta \in \Theta | \bar \theta \circ g(\theta) = \theta \}.
\]
It was shown in Theorem 2 in \cite{wu1983} that if $\Theta$ is open, $\phi$ and $\psi$ are differentiable and Assumption \ref{ass1} holds, then
\[
 \Gamma = \{\theta \in \Theta | \partial_\theta l(\theta) =0 \},
\]
such that the fixed points of the EM algorithm coincide with the stationary
points of $l$. In \cite{wu1983} it was proved that the set $\Gamma$ contains
all limit points of $(\theta_n)$ and that $(l(\theta_n))$ converges to
$l(\theta_0)$ for some $\theta_0 \in \Gamma$. In the following theorem we
extend these results to the FREM algorithm. Let $d(x,A)$ be the distance
between a point $x$ and a set $A$.

For the convergence of our forward-reverse based EM algorithm, we naturally
also need to guarantee convergence of the corresponding forward-reverse
estimators. This can be guaranteed by the following assumption (cf. also \cite[Section 4]{BS}).
\begin{assumption}
  \label{ass2}
  \begin{enumerate}
  \item For any multi-indices $\alpha,\beta \in \mathbb{N}_0^d$ with $|\alpha|
    + |\beta| \le 2$ and any index $i$ there are constants $C_1 =
    C_1(i,\alpha,\beta), C_2 = C_2(i, \alpha,\beta) > 0$ such
    that
    \begin{equation*}
      | \partial_x^\alpha \partial_y^\beta p_{i}(x,y) | \le C_1 \exp\left( -
        C_2 |x-y|^2 \right).
    \end{equation*}
  \item $S$ is twice differentiable in its arguments and both $S$ and its
    first and second derivatives are polynomially bounded.
  \end{enumerate}
\end{assumption}
Now we are ready to state as the main result of this section a general convergence theorem for the FREM algorithm.
\begin{theorem}\label{thm:FREMconv}
  Let $(K_n)_{n \in \mathbb{N}}$ satisfy \eqref{def:Kn} and choose $\theta_0
  \in K_0$. Suppose that Assumptions \ref{ass1} and \ref{ass2} hold and that
  $l(\Gamma)$ is compact, then the stable FREM random sequence $(\theta_n)_{n
    \geq 0}$ has the following properties:
 \begin{enumerate}
  \item $\lim_n p_n < \infty$ a.s. and $(\theta_n)$ is almost surely bounded.
  \item $\lim_n d(l(\theta_n),l(\Gamma))= 0$ almost surely.
  \item If also $l$ is $s$-times differentiable, then $\lim_n d(\theta_n,\Gamma)= 0$ almost surely.
 \end{enumerate}
\end{theorem}
\begin{proof}
 (1) Set
 \[
  g^{FR}(\theta) := \mathbb{E}_\theta^{FR} [S(X_0, \ldots, X_{n_N})| X_{n_0}= x_{n_0},\dots,X_{n_N}=x_{n_N}].
 \]
With the above notation an iteration of the FREM algorithm can be written as
\[
 \theta_{m+1} =  \bar \theta (g^{FR} (\theta_m)).
\]
 It was shown in Lemma 2 in \cite{delyon} that the incomplete data $\log$-likelihood $l$ is a natural Lyapunov function relative to $T$ and to the set of fixed points $\Gamma$.  If for any $\epsilon >0$ and compact $K \subset \Theta$ we have
 \begin{equation}\label{Tconv}
  \sum_m \mathbf{1}_{|l\circ \bar \theta (g^{FR} (\theta_m))- l \circ T(\theta_m)|\mathbf{1}_{\theta_m \in K}\geq \epsilon} < \infty \quad \text{ a.s.},
 \end{equation}
then Proposition 11 in \cite{fort} implies in our setting that $\lim \sup_n p_n < \infty$ almost surely and that $\theta_n$ is a compact sequence such that (1) follows. To obtain \eqref{Tconv} it is sufficient by Borel-Cantelli to prove that
 \[
   \sum_m P \left( |l\circ \bar \theta (g^{FR} (\theta_m))- l \circ T(\theta_m)|\mathbf{1}_{\theta_m \in K}\geq \epsilon\right) < \infty.
 \]
Define for any $\delta >0$ an $\epsilon$-neighborhood of $K$ by
\[
 K_\delta := \{ x \in \mathbb{R}^q |\inf_{z \in K} |z-x| \leq \delta \}.
\]
By assumption $l$ and $\bar \theta$ are continuous, such that for any $\delta >0$ there exists $\eta >0$ such that for any $x,y \in K_\delta$ we have $|l\circ \bar \theta (x) - l \circ \bar \theta (y)| \leq \epsilon$ whenever $|x-y| \leq \eta$.

Choosing now $\bar \epsilon = \delta \wedge \eta$ yields
\begin{align*}
 P \left( |l\circ \bar \theta (g^{FR} (\theta_m))- l \circ T(\theta_m)|\mathbf{1}_{\theta_m \in K}\geq \epsilon\right) = P \left( |l\circ \bar \theta (g^{FR} (\theta_m))- l \circ \bar \theta (g(\theta_m))|\mathbf{1}_{\theta_m \in K}\geq \epsilon\right)\\
 = P \left( |l\circ \bar \theta (g^{FR} (\theta_m))- l \circ \bar \theta (g(\theta_m))|\mathbf{1}_{\theta_m \in K}\geq \epsilon  ,|g^{FR} (\theta_m) - g(\theta_m)|\mathbf{1}_{\theta_m \in K} \leq \delta \right)\\
 + P \left( |l\circ \bar \theta (g^{FR} (\theta_m))- l \circ \bar \theta (g(\theta_m))|\mathbf{1}_{\theta_m \in K}\geq \epsilon  ,|g^{FR} (\theta_m) - g(\theta_m)|\mathbf{1}_{\theta_m \in K} \geq \delta \right)\\
  \leq 2 P\left( |g^{FR} (\theta_m) - g(\theta_m)|\mathbf{1}_{\theta_m \in K} \geq \bar \epsilon \right)
\end{align*}
Markov's inequality gives then
\[
 P \left( |l\circ \bar \theta (g^{FR} (\theta_m))- l \circ T(\theta_m)|\mathbf{1}_{\theta_m \in K}\geq \epsilon\right) \leq 2 \epsilon^{-k} \mathbb{E} \left[
  \left|g^{FR} (\theta_m) - g(\theta_m) \right|^k \mathbf{1}_{\theta_m \in K} \right]
\]
for some $k >0$.

By \cite[Theorem 4.18]{BS} (see also Remark~\ref{rem:FREM-convergence} below),
we can always choose a number $N$ of samples for the forward-reverse algorithm
and a corresponding bandwidth $\epsilon = \epsilon_N = N^{-\alpha}$ such that
\begin{equation*}
  \mathbb{E} \left[
  \left|g^{FR} (\theta_m) - g(\theta_m) \right|^2 \mathbf{1}_{\theta_m \in K}
\right] \le \f{C}{N}
\end{equation*}
for some constant $C$.
We note that the choice of $\alpha$ depends on the
dimension $d$ as well as on the \emph{order} of the kernel. For instance, for
$d \le 4$ and a standard first order accurate kernel $K$, we can choose any
$1/4 \le \alpha \le 1/d$. 

In any case, if we choose $N > m$ then
\begin{equation}\label{FRconv}
   \sum_m P \left( |l\circ \bar \theta (g^{FR} (\theta_m))- l \circ T(\theta_m)|\mathbf{1}_{\theta_m \in K}\geq \epsilon\right)< \infty,
\end{equation}
which proves (1).

To prove (2) and (3) observe that for every $K \subset \Theta$ we have
 \[
  \lim_m | l(\theta_{m+1}) - l \circ T( \theta_m) | \mathbf{1}_{\theta_m \in K} =0 \quad a.s. 
 \]
By Borel-Cantelli it is sufficient to prove that
\[
 \sum_m P (| l(\theta_{m+1}) - l \circ T( \theta_m) | \mathbf{1}_{\theta_m \in K} \geq \epsilon ) < \infty.
\]
But since we have shown in (1) that $p_n$ is finite a.s., we have in the above sum that $\theta_{m+1} = \bar \theta (g^{FR} (\theta_m))$ in almost all summands. Hence, it is sufficient to show that
\[
 \sum_m P (| l(\bar \theta (g^{FR} (\theta_m))) - l \circ T( \theta_m) | \mathbf{1}_{\theta_m \in K} \geq \epsilon ) < \infty,
\]
 which is nothing else than \eqref{FRconv}. The statement of (2) and (3) follows now from Sard's theorem (cf. \cite{broeckner}) and Proposition 9 in \cite{fort}.
\end{proof}

\begin{remark}
  \label{rem:FREM-convergence}
  In the above convergence proof we need to rely on the convergence proof of
  the forward-reverse estimator when the bandwidth tends to zero and the
  number of simulated Monte Carlo samples tends to infinity. Such a proof is
  carried out for the diffusion case in \cite[Theorem 4.18]{BS}, where also
  rates of convergence are given. We note that the proof only relies on the
  transition densities of (a discrete skeleton of) the underlying diffusion
  process. Hence, it immediately carries over to the present setting.
\end{remark}
Theorem \ref{thm:FREMconv} is a general convergence statement that links the limiting points of the FREM sequence to the set of stationary points of $l$. In many concrete models the set $\Gamma$ of stationary points consists of isolated points only such that an analysis of the Hessian of $l$ gives conditions for local maxima. A more detailed discussion in this direction can be found in \cite{delyon} for example.

\section{Implementation and complexity of the FREM algorithm}
\label{sec:impl-forrev}

Before presenting two concrete numerical examples, we will first discuss
general aspects of the implementation of the forward-reverse EM algorithm. For
this purpose, let us, for simplicity, assume that the Markov chains $X$ and
$(Y, \mathcal{Y})$ are time-homogeneous, i.e., that $p \equiv p_k$ and $q
\equiv q_k$ do not depend on time $k$. We assume that we observe the Markov
process $X$ at times $0 = i_0 < \cdots < i_r = N$, i.e., our data consist of
the values $X_{i_k} = x_{i_k}$, $k=0, \ldots, r$. For later use, we introduce
the shortcut-notation $\mathbf{x} := (x_{i_j})_{j=0}^r$.

The law of $X$ depends on an $s$-dimensional parameter $\theta \in \R^s$,
which we are trying to estimate, i.e., $p = p^\theta$. To this end, let
\begin{equation*}
  \ell(\theta; x_0, \ldots, x_N) := \sum_{i=1}^N \log p^\theta(x_{i-1}, x_i)
\end{equation*}
denote the log-likelihood function for the estimation problem assuming full
observation. As before, we make the structural assumption that
\begin{equation}
  \label{eq:structural}
  \ell(\theta; x_0, \ldots, x_N) = \phi(\theta) + \sum_{i=1}^n S_i(x_0,
  \ldots, x_N) \psi_i(\theta).
\end{equation}
For simplicity, we further assume that there are functions $S_i^j$ such that
\begin{equation*}
  S_i(x_0, \ldots, x_N) = \sum_{j=1}^r S^j_i(x_{i_{j-1}}, \ldots, x_{i_j}).
\end{equation*}
The structural assumption~(\ref{eq:structural}) allows us to effectively
evaluate the conditional expectation of the log-likelihood $\ell_c$ for
different parameters $\theta$, without having to re-compute the conditional
expectations. More precisely, recall that for a given guess
$\widetilde{\theta}$ the E step of the EM algorithm consists in calculating
the function
\begin{equation}
  \label{eq:Q-function-again}
  \theta \mapsto Q(\theta; \ttheta, \mathbf{x}) := \E_{\ttheta}\left[
    \left. \ell_c\left( \theta; X_0, \ldots, X_N) \right) \right| X_{i_j} =
    x_{i_j}, \ j=0, \ldots, r \right],
\end{equation}
with $\E_{\ttheta}$ denoting (conditional) expectation under the parameter
$\ttheta$. Inserting the structural assumption~(\ref{eq:structural}), we
immediately obtain
\begin{equation*}
  Q(\theta; \ttheta, \mathbf{x}) = \phi(\theta) + \sum_{i=1}^m \psi_i(\theta)
  \E_{\ttheta}\left[ 
    \left. S_i(X_0, \ldots, X_N) \right| X_{i_j} =
    x_{i_j}, \ j=0, \ldots, r \right] = \phi(\theta) + \sum_{i=1}^m
  S_i(\theta) z_i^{\ttheta} 
\end{equation*}
with $z_i^{\ttheta} := \E_{\ttheta}\left[ \left. S_i(X_0, \ldots, X_N) \right|
  X_{i_j} = x_{i_j}, \ j=0, \ldots, r \right]$, $i=1, \ldots, m$. Note that
the definition of $z_i^{\ttheta}$ does not depend on the free parameter
$\theta$. Thus, only one (expensive) round of calculations of conditional
expectations is needed for a given $\ttheta$, producing a cheap-to-evaluate
function in $\theta$, which can then be fed into any maximization algorithm.

For any given $\ttheta$, the calculation of the numbers $z_1^{\ttheta},
\ldots, z_m^{\ttheta}$ requires running the forward-reverse algorithm for
conditional expectations. More precisely, using the Markov property we decompose
\begin{multline*}
  z_i^{\ttheta} := \E_{\ttheta}\left[ \left. S_i(X_0, \ldots, X_N) \right|
    X_{i_j} = x_{i_j}, \ j=0, \ldots, r \right] \\
  = \sum_{j=1}^r \E_{\ttheta}\left[ \left. S_i^j(X_{i_{j-1}}, \ldots, X_{i_j})
    \right| X_{i_{j-1}} = x_{i_{j-1}}, X_{i_j} = x_{i_j} \right].
\end{multline*}
All these conditional expectations are of the Markov-bridge type for which the
forward-reverse algorithm is designed. Hence, for each iteration of the EM
algorithm, we apply the forward-reverse algorithm $r$ times, one for the
time-intervals $i_{j-1}, \ldots, i_j$, $j=1, \ldots, r$, evaluating all the
functionals $h_1^j, \ldots, h_m^j$ at one go.

\subsection{Choosing the reverse process}
\label{sec:choos-reverse-proc}

Recall the defining equation for the one-step transition density $q$ of the
reverse process given in~\eqref{varfi}. For simplicity, we shall again assume
that the forward and the reverse processes are time-homogeneous, implying
that~\eqref{varfi} can be re-expressed as
\begin{equation*}
  q(y,z) = \frac{p(z,y)}{\psi(y,z)}.
\end{equation*}
Notice that in this equation only $p$ is given a-priori, i.e., the user is
free to choose any re-normalization $\psi$ provided that for any $y \in \R^d$
the resulting function $z \mapsto q(y,z)$ is non-negative and integrates to
$1$. In particular, we can turn the equation around, choose \emph{any}
transition density $q$ and \emph{define}
\begin{equation*}
  \psi(y,z) := \frac{p(z,y)}{q(y,z)}.
\end{equation*}
Note, however, that for the resulting forward-reverse process square
integrability of the process $\mathcal{Y}$ is desirable. More precisely, only
square integrability of the (numerator of the) complete estimator
corresponding to (\ref{eq:for-rev-intro}) is required, but it seems
far-fetched to hope for any cancellations giving square integrable estimators
when $\mathcal{Y}$ itself is not square integrable. From a practical point of
view, it therefore seems reasonable to aim for functions $\psi$ satisfying
\begin{equation*}
  \psi \approx 1
\end{equation*}
in the sense that $\psi$ is bounded from above by a number
slightly smaller than $1$ and bounded from below by a number slightly smaller
than $1$. Indeed, note that $\mathcal{Y}$ is obtained by multiplying terms of
the form $\psi(Y_n, Y_{n+1})$ along the whole trajectory of the reverse
process $Y$. Hence, if $\psi$ is bounded by a large constant, $\mathcal{Y}$
could easily take extremely large values, to the extent that buffer-overflow
might occur in the numerical implementation -- think of multiplying $100$
numbers of order $100$. On the other hand, if $\psi$ is considerably smaller
than $1$, $\mathcal{Y}$ might take very small values, which can cause problems
in particular taking into account the division by the forward-reverse
estimator for the transition density in the denominator of the forward-reverse
estimator.

Heuristically, the following procedure seems promising.
\begin{itemize}
\item If $y \mapsto \int_{\R^d} p(z,y) dz$ can be computed in closed form (or
  so fast that one can think of a closed formula), then choose
  \begin{equation*}
    \psi(y) := \psi(y,z) = \int_{\R^d} p(z,y) dz.
  \end{equation*}
\item Otherwise, assume that we can find a non-negative (measurable) function
  $\widetilde{p}(z,y)$ with closed form expression for $\int_{\R^d}
  \widetilde{p}(z,y) dz$ such that $p(z,y) \approx \widetilde{p}(z,y)$. Then
  define
  \begin{equation*}
    q(y,z) := \frac{\widetilde{p}(z,y)}{\int_{\R^d} \widetilde{p}(z,y) dz},
  \end{equation*}
  which is a density in $z$. By construction, we have
  \begin{equation*}
    \psi(y,z) = \frac{p(z,y)}{q(y,z)} = \int_{\R^d} \widetilde{p}(z,y) dz
    \, \frac{p(z,y)}{\widetilde{p}(z,y)},
  \end{equation*}
  implying that we are (almost) back in the first situation.
\end{itemize}

\begin{remark}
  Even if we can, indeed, explicitly compute $\psi(y,z) = \int_{\R^d} p(z,y)
  dz$, there is generally no guarantee that $\mathcal{Y}$ has (non-exploding)
  finite second moments. However, in practice, this case seems to be much
  easier to control and analyze.
\end{remark}

\subsection{Complexity of the forward-reverse algorithm}
\label{sec:compl-forw-reverse}

We end this general discussion of the forward-reverse EM algorithm by a
refined analysis of the complexity of the forward-reverse algorithm for
conditional expectations as compared to \cite{BS}. We start with an auxiliary
lemma concerning the maximum product of numbers of two species of balls in
bins, which is an easy consequence of a result by Gonnet~\cite{gon81}, see
also~\cite[Section 8.4]{sed}.
\begin{lemma}
  \label{lem:ball-bin}
  Let $X$ be a random variable supported in a compact set $D \subset \R^d$
  with a uniformly bounded density $p$. For any $K \in \N$ construct a
  partition $B_1^K, \ldots, B_K^K$ of $D$ in measurable sets of equal Lebesgue
  measure $\lambda(B^K_i) = \lambda(D)/K$, $i = 1, \ldots, K$. Finally, for
  given $N \in \N$ let $X_1, \ldots, X_N$ be a sequence of independent copies
  of $X$ and define
  \begin{equation*}
    N_k \coloneqq \# \left\{\left.i \in \set{1, \ldots, N} \,\right| \, X_i
      \in B^K_k \right\}, \quad k = 1, \ldots, K.
  \end{equation*}
  For $N, K \to \infty$ such that $N = \mathcal{O}(K)$ we have the asymptotic
  relation
  \begin{equation*}
    \mathbb{E} \left[ \max_{k=1, \ldots, K} N_k \right] = \mathcal{O}\left(
      \f{\log N}{\log \log N} \right).
  \end{equation*}
\end{lemma}
\begin{proof}
  Let
  \begin{equation*}
    p_k \coloneqq P\left( X \in B^K_k \right) \le \norm{p}_{\infty}/K
  \end{equation*}
  and observe that the random vector $\left( N_1, \ldots, N_K \right)$
  satisfies a multi-nomial distribution with parameters $K, N$ and $(p_1,
  \ldots, p_K)$.

  The proof for the statement in the special case of $p_1 = \ldots = p_K =
  1/K$ is given in Gonnet~\cite{gon81}, so we only need to argue that the
  relation extends to the non-uniform case. To this end, let $K' \coloneqq
  \floor{K/\norm{p}_\infty}$ and let $(M_1, \ldots, M_{K'})$ denote a
  multi-nomial random variable with parameters $N$, $K'$ and $(1/K', \ldots,
  1/K')$. As $N = \mathcal{O}(K')$ we have by Gonnet's result that
  \begin{equation*}
    \mathbb{E}\left[ \max_{k=1, \ldots, K'} M_k \right] = \mathcal{O}\left(
      \f{\log N}{\log\log N} \right).
  \end{equation*}
  Moreover, it is clear that $\mathbb{E}\left[ \max_{k=1, \ldots, K} N_k
  \right] \le \mathbb{E}\left[ \max_{k=1, \ldots, K'} M_k \right]$ and we have
  proved the assertion.
\end{proof}

\begin{theorem}
  \label{thr:complexity-fr}
  Assume that the transition densities $p$ and $q$ have compact support in
  $\R^d$.\footnote{Obviously, this assumption can be weakened.} Moreover,
  assume that the kernel $K$ is supported in a ball of radius $R > 0$. Then
  the forward-reverse algorithm for $N$ forward and reverse trajectory based
  on a bandwidth proportional to $N^{-1/d}$ can be implemented in such a way
  that its expected cost is $\mathcal{O}(N \log N)$ as $N \to \infty$.
\end{theorem}
\begin{proof}
  In order to increase the clarity of the argument, we re-write the double sum
  in the forward-reverse algorithm to a simpler form, which highlights the
  computational issues. Indeed, we are trying to compute a double sum of the
  form
  \begin{equation}
    \label{eq:double-sum-simplified}
    \sum_{i=1}^N \sum_{j=1}^N F_{i,j} K_\epsilon\left(X^i_{n^\ast} -
      Y^j_{\hat{n}_l} \right),
  \end{equation}
  where $F_{i,j}$ obviously depends on the whole $i$th sample of the forward
  process $X$ and on the whole $j$th sample of the reverse process $(Y,
  \mathcal{Y})$.

  We may assume that the end points $X^i_{n^\ast}$ and $Y^j_{\hat{n}_l}$ of
  the $N$ samples of the forward and reverse trajectories are contained in a
  compact set $[-L,L]^d$. (Indeed, the necessary re-scaling operation can
  obviously be done with $\mathcal{O}(N)$ operations.) In fact, for ease of
  notation we shall assume that the points are actually contained in
  $[0,1]^d$. We sub-divide $[0,1]^d$ in boxes with side-length $S \epsilon$,
  where $S > R$ is chosen such that $1/(S \epsilon) \in \mathbb{N}$. Note that
  there are $K := (S \epsilon)^{-d}$ boxes which we order lexicographically
  and associate with the numbers $1, \ldots, K$ accordingly.

  In the next step, we shall order the points $X^i_{n^\ast}$ and
  $Y^j_{\hat{n}_l}$ into these boxes. First, let us define a function $f_1:
  [0,1]^d \to \{1, \ldots, 1/(S\epsilon)\}^d$ by setting
  \begin{equation*}
    f_1(x) := \left( \lceil x_1/(S\epsilon) \rceil, \ldots, \lceil x_d/(S
      \epsilon) \rceil \right),
  \end{equation*}
  with $\lceil \cdot \rceil$ denoting the smallest integer larger or equal
  than a number. Moreover, define $f_2: \{1, \ldots, 1/(S\epsilon)\}^d \to
  \{1, \ldots, K\}$ by
  \begin{equation*}
    f_2(i_1, \ldots, i_d) := (i_1-1) (S \epsilon)^{-d+1} + (i_2-1) (S
    \epsilon)^{-d+2} + \cdots + (i_d-1) + 1.
  \end{equation*}
  Obviously, a point $x \in [0,1]^d$ is contained in the box number $k$ if and
  only if $f_2(f_1(x)) = k$.\footnote{To make this construction fully
    rigorous, we would have to make the boxes half-open and exclude the
    boundary of $[0,1]^d$.}
  Now we apply a sorting algorithm like quick-sort to both sets of points
  $\left(X^1_{n^\ast}, \ldots, X^N_{n^\ast} \right)$ and $\left(
    Y^1_{\hat{n}_l}, \ldots, Y^N_{\hat{n}_l} \right)$ using the ordering
  relation defined on $[0,1]^d \times [0,1]^d$ by
  \begin{equation*}
    x < y : \iff f_2(f_1(x)) < f_2(f_1(y)).
  \end{equation*}
  Sorting both sets incurs a computational cost of $\mathcal{O}(N \log N)$, so
  that we can now assume that the vectors $X^i_{n^\ast}$ and $Y^i_{\hat{n}_l}$
  are ordered.

  Notice that $K_{\epsilon}(x-y) \neq 0$ if and only if $x$ and $y$ are
  situated in neighboring boxes, i.e., if $\abs{f_1(x) - f_1(y)}_{\infty} \le
  1$, where we define $\abs{\alpha}_\infty := \max_{i=1, \ldots, d}
  \abs{\alpha_i}$ for multi-indices $\alpha$. Moreover, there are $3^d$ such
  neighboring boxes, whose indices can be easily identified, in the sense that
  there is a simple set-valued function $f_3$ which maps an index $k$ to the
  set of all the indices $f_3(k)$ of the at most $3^d$ neighboring
  boxes. Moreover, for any $k \in \{1, \ldots, K\}$ let $X^{i(k)}_{n^\ast}$ be
  the first element of the ordered sequence of $X^i_{n^\ast}$ lying in the box
  $k$. Likewise, let $Y^{j(k)}_{\hat{n}_l}$ be the first element in the
  ordered sequence $Y^j_{\hat{n}_l}$ lying in the box with index $k$. Note
  that identifying these $2K$ indices $i(1), \ldots, i(K)$ and $j(1), \ldots,
  j(K)$ can be achieved at computational costs of order $\mathcal{O}(K\log N)
  = \mathcal{O}(N \log N)$.

  After all these preparations, we can finally express the double
  sum~\eqref{eq:double-sum-simplified} as
  \begin{equation}
    \label{eq:double-sum-computational}
    \sum_{i=1}^N \sum_{j=1}^N F_{i,j} K_\epsilon\left(X^i_{n^\ast} -
      Y^j_{\hat{n}_l} \right) = \sum_{k =1}^K \sum_{r \in f_3(k)}
    \sum_{i=i(k)}^{i(k+1)-1} \sum_{j=j(r)}^{j(r+1)-1} F_{i,j}
    K_\epsilon\left(X^i_{n^\ast} - Y^j_{\hat{n}_l} \right).
  \end{equation}
  Regarding the computational complexity of the right hand side, note that we
  have the deterministic bounds
  \begin{gather*}
    K = \mathcal{O}(N),\\
    \abs{f_3(k)} \le 3^d.
  \end{gather*}
  Moreover, regarding the stochastic contributions, the expected maximum
  number of samples $X^i_{n^\ast}$ ($Y^j_{\hat{n}_l}$, respectively) contained
  in any of the boxes is bounded by $\mathcal{O}(\log N)$ by
  Lemma~\ref{lem:ball-bin}, i.e.,
  \begin{equation*}
    \mathbb{E}\left[ \max_{r=1, \ldots, K} (j(r+1) - j(r)) \right]
    \mathcal{O}(\log N),
  \end{equation*}
  which then needs to be multiplied by the total number $N$ of points
  $X^i_{n^\ast}$ to get the complexity of the double summation step.
\end{proof}

\begin{remark}
  As becomes apparent in the proof of Theorem~\ref{thr:complexity-fr}, the
  constant in front of the asymptotic complexity bound does depend
  exponentially on the dimension $d$.
\end{remark}

\begin{remark}
  Notice that the box-ordering step can be omitted by maintaining a list of
  all the indices of trajectories whose end-points lie in every single
  box. The asymptotic rate of complexity of the total algorithm does not
  change by omitting the ordering, though.
\end{remark}

\section{Applications of the FREM algorithm} \label{sec:appl}

The forward reverse EM algorithm is a versatile tool for parameter estimation in dynamic stochastic models. It can be applied in discrete time Markov models, but also in the the setting of discrete observations of time-continuous Markov processes such as diffusions for example. 

In this section we give examples from both worlds: we start by a discretized Ornstein-Uhlenbeck process that serves as a benchmark model, since the likelihood function can be treated analytically. Then we give an example of a partially hidden Markov model that is motivated be applications in system biology in \cite{langrock}. Finally, we remark on limitations of the EM algorithm when the log-likelihood is not integrable and demonstrate these limitations in the context of a Cox-Ingersoll-Ross process. For a complex real data application of our method we refer the interested reader to the forthcoming companion paper \cite{BMTV}.

\subsection{Ornstein-Uhlenbeck dynamics}
\label{sec:ou}

In this section we apply the forward-reverse EM algorithm to simulated data from a discretized Ornstein-Uhlenbeck process. The corresponding Markov chain is thus given by
\begin{equation} \label{eq:ou}
  X_{n+1} = X_n + \lambda  X_n \dt + \dW_{n+1},\quad n\geq 0
\end{equation}
where $W_n$ are independent random variables distributed according to $N(0,\Delta t)$. The drift parameter $\lambda \in \mathbb{R}$ is unknown and we will employ the forward reverse EM algorithm to estimate it from simulated data. The Ornstein-Uhlenbeck model has the advantage that the likelihood estimator is available in closed form and we can thus compare it to the results of the EM algorithm.

In each simulation run we suppose that we have known observations \[X_{0}, X_{10\Delta t}, \ldots,X_{40 \Delta t}\]
for varying step size $\Delta t$ and use the EM methodology to approximate the likelihood function in between. We perform six iteration of the algorithm with increasing number of data points $N$.

In table \ref{tableOUsim1} we summarize the results of two runs for the discrete Ornstein-Uhlenbeck chain. The mean and standard deviation are estimated from 1000 Monte Carlo iterations.  We find that already after three steps the mean  is very close to the corresponding estimate of the true MLE. This indicates a surprisingly fast convergence for this example. Note also that the approximated value of the likelihood function stabilizes extremely fast at the maximum.

\begin{table}
\begin{tabular}{ c c c c c c c}

  $\Delta t$ &N &bandwidth & mean $\hat \lambda$ & std dev $\hat \lambda$ & likel. & std dev likel. \\
\toprule
 0.1 & 2000&0.0005 & 0.972 & 0.0135 &  -3.402 & 0.00290\\
&8000&0.000125&1.098&0.00841&-3.383&0.00062\\
&32000& 3.125e-05& 1.132 &0.00476 &-3.381 &0.000123\\
&128000&7.812e-06 & 1.151&0.00236& -3.381& 2.783e-05\\
&512000&1.953e-06&  1.157& 0.00117& -3.381& 4.745e-06\\
&2048000&4.882e-07 & 1.159 &0.000581& -3.381& 1.005e-06\\
\midrule
 0.05&2000 &0.0005 &1.160 &0.0141 &-3.107 &0.000854 \\
&8000 &0.000125 &1.247 &0.00872 &-3.103 &9.867e-05 \\
&32000 &3.125e-05 &1.253 &0.00468 &-3.103 &1.329e-05 \\
&128000 &7.812e-06  &1.265 &0.00225 &-3.103 &3.772e-06 \\
&512000 &1.953e-06 &1.265 &0.00111 &-3.103& 6.005e-07 \\

\end{tabular}
\caption{Behavior of the forward-reverse EM algorithm for a discretized
  Ornstein-Uhlenbeck model for different step sizes $\Delta t$, initial guess
  $\lambda = 0.5$ and true MLE $\hat \lambda_{\text{MLE}} = 1.161$ and
  $1.266$, respectively.}
\label{tableOUsim1}
\end{table}

Table \ref{tableOUsim2} gives results for the same setup as in Table \ref{tableOUsim1} but with initial guess $\lambda = 2$ such that the forward-reverse EM algorithm converges from above to the true maximum of the likelihood function. We observe that the smaller step size $\Delta t = 0.05$ results in a more accurate approximation of the likelihood and also of the true MLE. It seems that the step size has crucial influence on the convergence rate of the algorithm, since for $\Delta t = 0.05$ the likelihood stabilizes already from the second iteration.

\begin{table}
\begin{tabular}{ c c c c c c c}

  $\Delta t$ &N &bandwidth & mean $\hat \lambda$ & std dev $\hat \lambda$ & likel. & std dev likel. \\
\toprule
 0.1 & 2000&0.0005 & 1.554 &0.0353 &-3.457 &0.0134\\
&8000&0.000125&1.312& 0.0127& -3.393& 0.00221\\
&32000& 3.125e-05& 1.217& 0.00544& -3.382& 0.000351\\
&128000&7.812e-06 & 1.185 &0.00245& -3.381& 5.817e-05\\
&512000&1.953e-06&  1.168& 0.00121& -3.381& 1.227e-05 \\

\midrule
 0.05&2000 &0.0005 &1.390& 0.0248& -3.108& 0.00238 \\
&8000 &0.000125 & 1.289&0.00925& -3.103& 0.000130 \\
&32000 &3.125e-05 &1.261& 0.00471& -3.103& 1.451e-05  \\
&128000 &7.812e-06  &1.266& 0.00221& -3.103& 2.538e-06 \\
&512000 &1.953e-06 &1.266& 0.00113& -3.103& 5.855e-07  \\

\end{tabular}
\caption{Behavior of the forward-reverse EM algorithm for a discretized
  Ornstein-Uhlenbeck model for different step sizes $\Delta t$, initial guess
  $\lambda = 2$ and true MLE $\hat \lambda_{\text{MLE}} = 1.161$ and $1.266$,
  respectively.}
\label{tableOUsim2}
\end{table}

In Figure \ref{OUhist1} the empirical distribution of 1000 estimates for $\lambda$ is plotted. The initial value was $0.5$ and the true maximum of the likelihood function is at $1.161$. The step size between observations was chosen to be $\Delta t = 0.1$. The histogram on the left shows the estimates after only one iteration and on the right the estimates were obtained from five iterations of the forward-reverse EM algorithm.

Figure \ref{OUhist2} depicts the distribution of 1000 Monte Carlo samples of the likelihood values that led to the estimates in Figure \ref{OUhist1}. It is interesting to see that after one iteration of the algorithm the likelihood values are approximately bell shaped (left histogram) whereas after five iterations the distributions becomes more and more one-sided as would be expected, since the EM algorithm only increase the likelihood from step to step towards the maximum.

\begin{figure}[htb]
\centering
\includegraphics[scale=0.4]{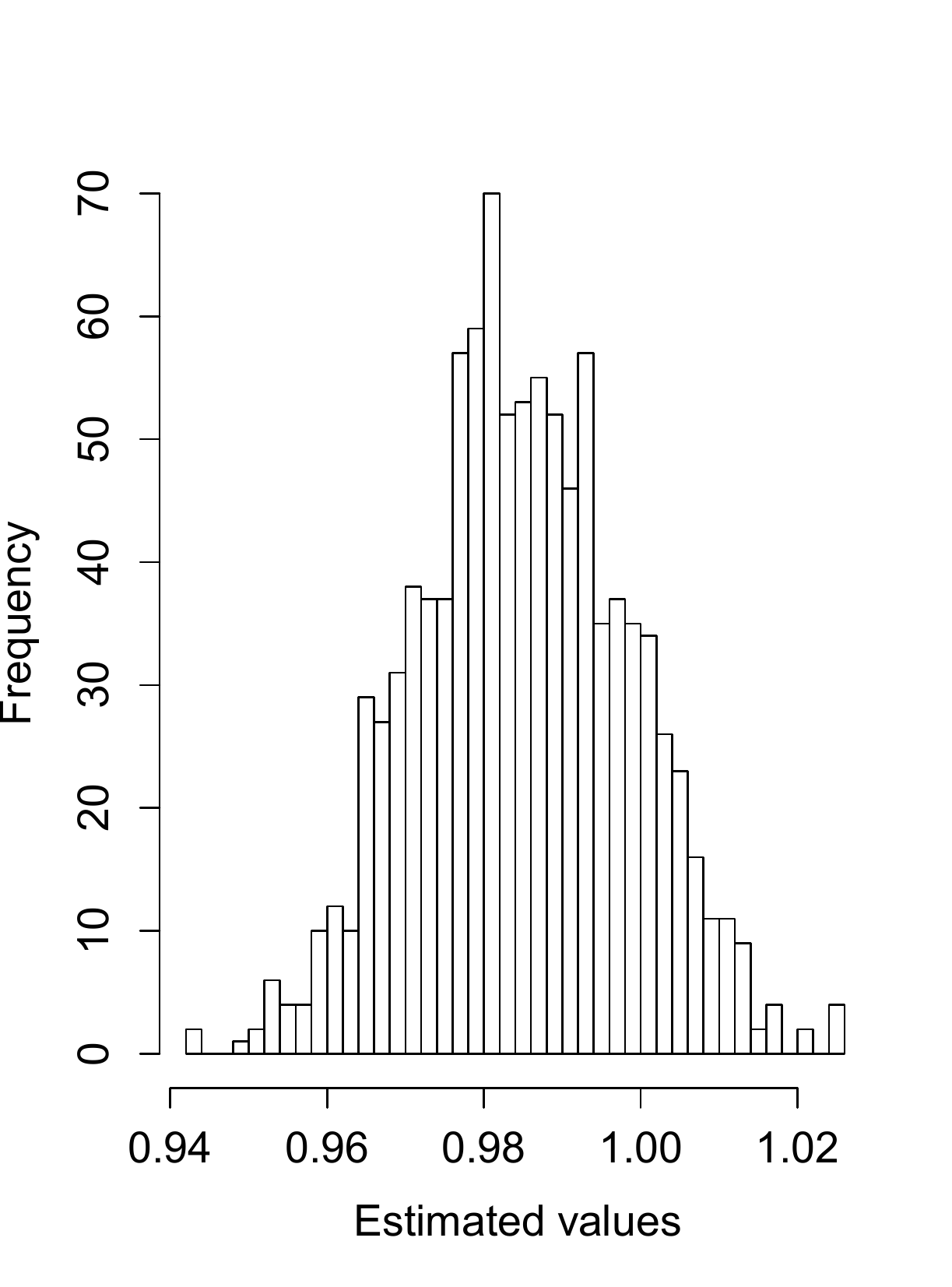}
\includegraphics[scale=0.4]{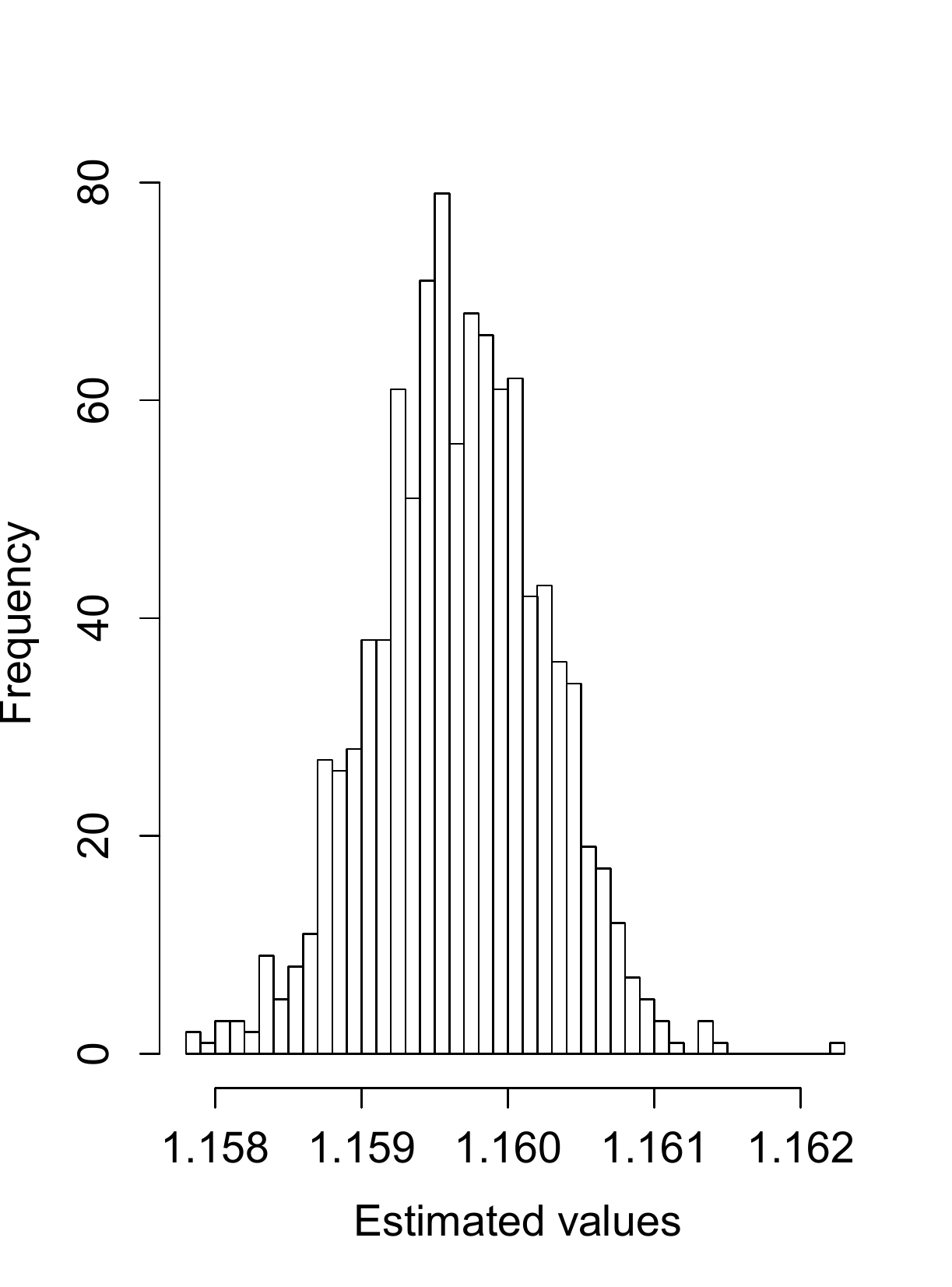}
\caption{Empirical distribution of 1000 estimates after one iteration (right) and after five iteration (left) of the forward-reverse EM algorithm.} \label{OUhist1}
\end{figure}

\begin{figure}[htb]
\centering
\includegraphics[scale=0.4]{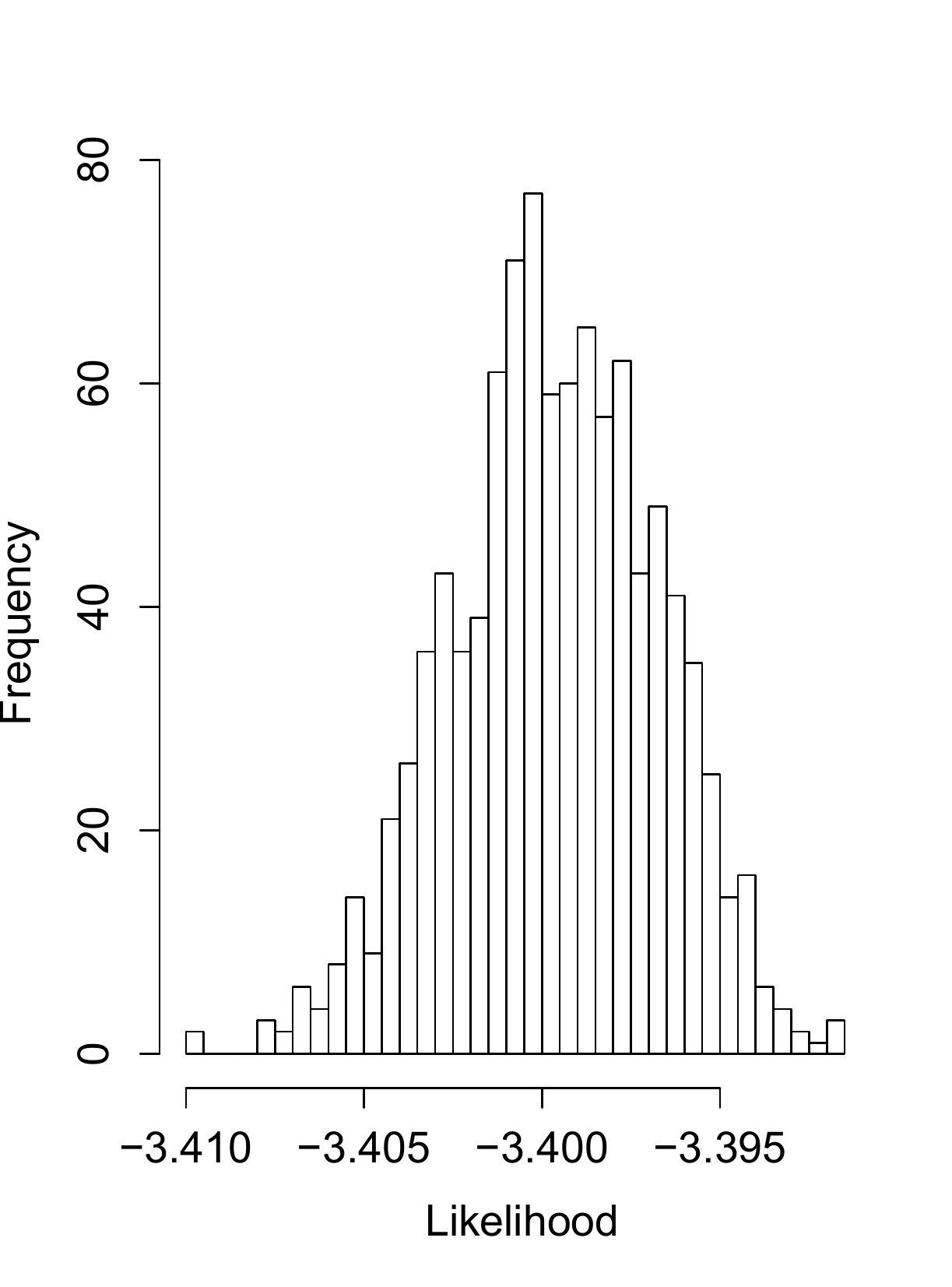}
\includegraphics[scale=0.4]{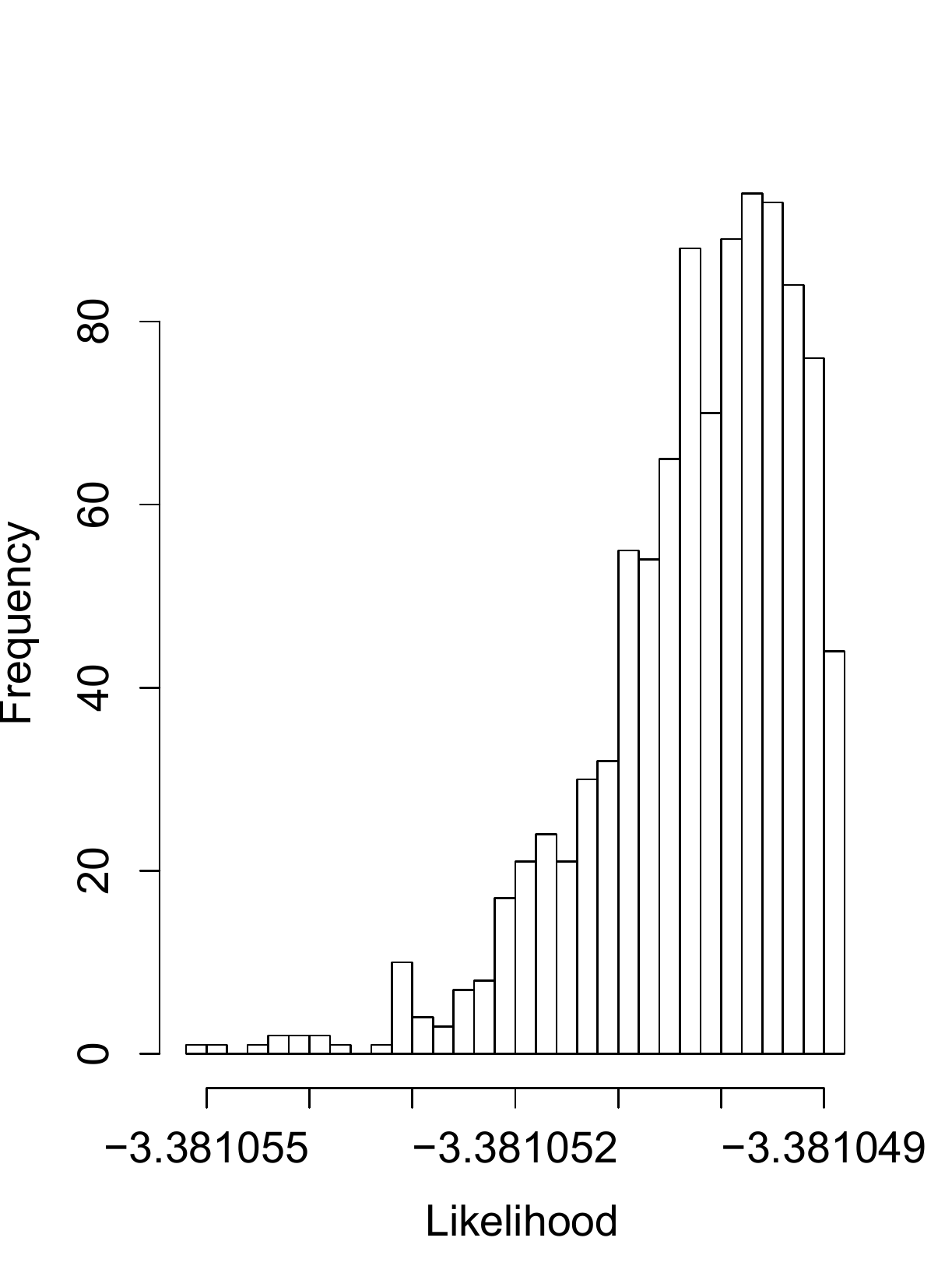}
\caption{Empirical distribution of the likelihood values of 1000 Monte Carlo samples after one iteration (right) and after five iteration (left) of the forward-reverse EM algorithm.} \label{OUhist2}
\end{figure}

Figure \ref{OUbox} shows the convergence of the forward reverse EM algorithm when the number of iterations increases. We find that already after 4 iterations the estimate is very close to the true MLE for $\lambda$. After six iterations the algorithm has almost perfectly stabilized at the value of the true MLE $\lambda = 1.16$.

\begin{figure}[htb]
\centering
\includegraphics[scale=0.7]{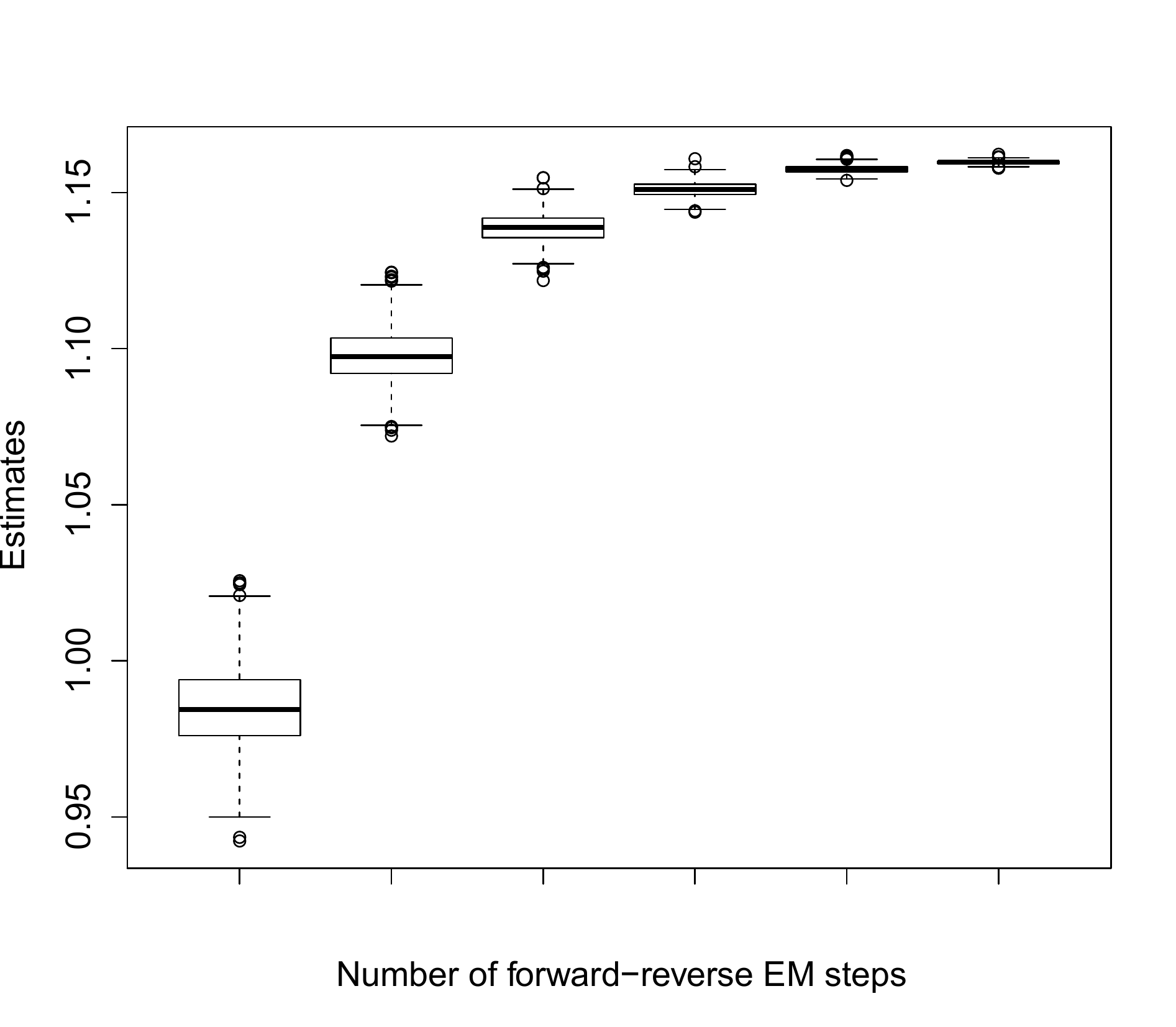}

\caption{Convergence of the forward-reverse EM algorithm from one to six iterations for each 1000 estimates of $\lambda$. The value of the true MLE is $\hat \lambda=1.161$.} \label{OUbox}
\end{figure}

\subsection{Hidden Markov models}

Our forward reverse EM algorithm can also be applied in the context of hidden Markov models (HMM). A typical HMM consists of observed state process and a hidden (i.e. unobserved) Markov chain that models the internal regime switching of the state process (see for example \cite{macdonald} for a general introduction). When for the hidden process only some values can be observed we speak of a partially HMM.

Recently, partially HMMs have been applied with some success in the modeling of ecological systems to understand the relation between animal populations and environmental parameters in a dynamical setting. The example considered here is a partially hidden Markov model and is inspired by applications in modeling mark-recapture-recovery data as discussed in \cite{langrock} for example. In mark-recapture-recovery studies every individual of a population is marked by some tag or ring that uniquely identifies the individual and some properties of interest are measured (e.g. weight, age etc). In subsequent surveys each individual is either recaptured such that the measurements could be taken again or if it can not be recaptured this will be recorded as a missing data point. Hence, we obtain a sequence of data with missing values that leads naturally to a partially HMM.

Let us consider a two-dimensional Markov chain $X_n=(X_n^1, X_n^2)$. For the first component $N$ observations $X_0^1, \ldots, X_N^1$ are given, whereas the second component $X^2$ is only observed partially after each $k \in \{1, \ldots, N\}$ time steps, i.e.
\[
	X_0^2, X_k^2, X_{2k}^2, \ldots, X_{\lfloor \frac{N}{k}\rfloor k}^2
\]
are known. Suppose that the one-step transition probabilities are normal distributions:
\[
\mathcal{L} (X_{n+1} | X_n) = N( \mu^\theta (X_n), \Sigma),
\]
with mean given by $\mu^\theta (x) = \theta + x^2$ for an unknown parameter $\theta \in \mathbb{R}^2$. The covariance matrix $\Sigma \in \mathbb{R}^{2 \times 2}$ is assumed to be constant.

In this setup the $\log$-likelihood function based on full observations $X$ is given by
\[
l_c (\theta, X_0, \ldots, X_N) = 1- 2 \pi \det (\Sigma)^{1/2} - \frac{1}{2} \sum_{l = 1}^{N} Y_l^\theta.
\]
where
\begin{align*}
  Y_l^\theta = \left( X_l - \mu^\theta (X_{l-1})\right)^\top \Omega \left( X_l - \mu^\theta (X_{l-1})\right)
\end{align*}
with precision matrix $\Omega = (\omega_{ij}) := \Sigma^{-1}$. In coordinates we thus have
\[
	Y_l^\theta = \omega_{11} (U_{l,1})^2 + (2\omega_{12}) U_{l,1} U_{l,2} + \omega{22} U_{l,2}^2,
\]
where $ U_{l,i} := X_l^i - \theta_i - (X_{l-1}^i)^2$. Calculating the score function gives therefore
\[
\frac{\partial l_c}{\partial \theta_1} = - \frac{1}{2} \sum_{l = 1}^{N}\omega_{11} \left( 2(X_l^1)^2 + \theta^1 + 2 (X_{l-1}^1)^2 \right) + 2 \omega_{12} \left( \theta^2 - X_l^2 + (X_{l-1}^2)^2 \right)
\]
and
\[
	\frac{\partial l_c}{\partial \theta_2} = - \frac{1}{2} \sum_{l = 1}^{N}\omega_{22} \left( 2(X_l^2)^2 + \theta^2 + 2 (X_{l-1}^1)^2 \right) + 2 \omega_{12} \left( \theta^1 - X_l^1 + (X_{l-1}^1)^2 \right).
\]
Since not all values of $X^2$ are observed, this $\log$-likelihood cannot be maximized directly. Instead, the forward reverse algorithm approximates the expected $\log$-likelihood given the partial observations. The E-step in this model reads as follows.

\subsection*{FR E-step}
Evaluate the forward-reverse approximation
\[
  Q_m (\theta, \theta_m, X_1, \ldots, X_N) = \mathbb{E}_{\theta_m^{FR}} \left[ l_c (\theta, X_0, \ldots, X_n) | X_0^1, \ldots, X_N^1; X_{i k}^2, i =0,\ldots, \lfloor N/k\rfloor \right].
\]
This can be rewritten as
\[
Q_m (\theta, \theta_m, X_1, \ldots, X_N) = \operatorname{const} - \frac{1}{2}  \mathbb{E}_{\theta_m}^{FR} \left[ \sum_{l =1}^N Y_l^\theta | X_0^1, \ldots, X_N^1; X_{i k}^2, i =0,\ldots, \lfloor N/k\rfloor \right].
\]
By approximating directly the score function this can be further simplified to
\begin{align*}
	Q_m^{\partial \theta_1} &(\theta, \theta_m, X_1, \ldots, X_N) =  - \frac{1}{2}(\omega_{11} \theta^1 +2\omega_{12} \theta^2) \\
	\mathbb{E}_{\theta_m}^{FR} &\left[ \sum_{l =1}^N 2\omega_{11}((X_l^2)^2 + (X_{l-1}^1)^2) 2\omega_{12} ((X_{l-1}^2)^2- X_l^2) | X_0^1, \ldots, X_N^1; X_{i k}^2, i =0,\ldots, \lfloor N/k\rfloor \right].
\end{align*}
and
\begin{align*}
	Q_m^{\partial \theta_2} &(\theta, \theta_m, X_1, \ldots, X_N) =  - \frac{1}{2}(\omega_{22} \theta^2 +2\omega_{12} \theta^1) \\
	\mathbb{E}_{\theta_m}^{FR} &\left[ \sum_{l =1}^N 2\omega_{22}((X_l^2)^2 + (X_{l-1}^2)^2) 2\omega_{12} ((X_{l-1}^1)^2- X_l^1) | X_0^1, \ldots, X_N^1; X_{i k}^2, i =0,\ldots, \lfloor N/k\rfloor \right].
\end{align*}
Due to the linearity of the score function in $\theta$ the M-step is now straightforward.
\subsection*{M-step}
Solve the linear system
\begin{align}
	Q_m^{\partial \theta_1} &(\theta, \theta_m, X_1, \ldots, X_N)=0 \\
	Q_m^{\partial \theta_2} &(\theta, \theta_m, X_1, \ldots, X_N)=0
\end{align}
for $\theta_1$ and $\theta_2$ to update $\theta_{m+1} = (\theta_1, \theta_2)$.

\begin{remark}
 The example given here can also be extended to HMMs with more involved likelihood structure. In particular, the case of a non linear score function in $\theta$ can easily be treated with standard numerical methods such that also in these examples the M-step remains feasible (see also \cite{liu} and \cite{meng1993}).
\end{remark}

\subsection{A final note of warning by a discrete Cox-Ingersoll-Ross example}

\label{sec:cir}

Consider the Markov chain given by
\begin{equation}
X_{n+1}=X_{n}+\lambda\left(  \theta-X_{n}\right)  \Delta t+\sigma\left\vert
X_{n}\right\vert ^{\gamma}\Delta W_{n+1},\label{eq:cir}%
\end{equation}
where $\Delta t$ is fixed and $\Delta W_{n}$ are independent random variables
distributed according to $\mathcal{N}(0,\Delta t)$. Moreover, we assume that
$0\leq\gamma$ is fixed and known. The other parameters $\sigma$, $\lambda$ and
$\theta$ are unknown and need to be estimated. In the case
$\gamma=1/2$ it corresponds a of Euler discretization of the
Cox-Ingersoll-Ross model from finance.

Up to constant terms (in the un-known parameters $\sigma$, $\lambda$ and
$\theta$), the log-likelihood function of a sequence of observations
$\mathbf{x}=(x_{0},\ldots,x_{N})$ of the full path of the process $X$ is given
by
\begin{align*}
\ell_{c}\left(  \sigma,\lambda,\theta;\mathbf{x}\right)   &  =\log\left(
\prod_{i=1}^{N}p(x_{i-1},x_{i})\right)  \\
&  =-N\log\sigma-\frac{1}{2\sigma^{2}\Delta t}\sum_{i=1}^{N}\frac{\left(
x_{i}-(1-\lambda\Delta t)x_{i-1}-\lambda\theta\Delta t\right)  ^{2}%
}{\left\vert x_{i-1}\right\vert ^{2\gamma}}\\
&  =-N\log\sigma-\frac{1}{2\sigma^{2}\Delta t}\sum_{i=1}^{N}\Biggl[\frac
{x_{i}^{2}}{\left\vert x_{i-1}\right\vert ^{2\gamma}}-2(1-\lambda\Delta
t)\frac{x_{i}x_{i-1}}{\left\vert x_{i-1}\right\vert ^{2\gamma}}\\
&  \quad-2\lambda\theta\Delta t\frac{x_{i}}{\left\vert x_{i-1}\right\vert
^{2\gamma}}+(1-\lambda\Delta t)^{2}\frac{x_{i-1}^{2}}{\left\vert
x_{i-1}\right\vert ^{2\gamma}}\\
&  \quad+2\lambda\theta\Delta t(1-\lambda\Delta t)\frac{x_{i-1}}{\left\vert
x_{i-1}\right\vert ^{2\gamma}}+\lambda^{2}\theta^{2}\Delta t^{2}\frac
{1}{\left\vert x_{i-1}\right\vert ^{2\gamma}}\biggr].
\end{align*}
Assume that we have given partial observations $X_{i_{0}},\ldots,X_{i_{r}}$ with
$i_{0}=0<\cdots<i_{r}=N$, while the remaining points $X_{j}$, $j\notin
\{i_{0},\ldots,i_{r}\}$, are assumed to be unobserved. Define random variables
$Z_{0}:=N$ and
\begin{align*}
&  Z_{1}:=\sum_{i=1}^{N}\frac{X_{i}^{2}}{\left\vert X_{i-1}\right\vert
^{2\gamma}}, & Z_{2} &  :=\sum_{i=1}^{N}\frac{X_{i}}{\left\vert X_{i-1}%
\right\vert ^{2\gamma}},\\
&  Z_{3}:=\sum_{i=1}^{N}\frac{X_{i-1}X_{i}}{\left\vert X_{i-1}\right\vert
^{2\gamma}}, & Z_{4} &  :=\sum_{i=1}^{N}\frac{1}{\left\vert X_{i-1}\right\vert
^{2\gamma}},\\
&  Z_{5}:=\sum_{i=1}^{N}\frac{X_{i-1}}{\left\vert X_{i-1}\right\vert
^{2\gamma}}, & Z_{6} &  :=\sum_{i=1}^{N}\frac{X_{i-1}^{2}}{\left\vert
X_{i-1}\right\vert ^{2\gamma}}.
\end{align*}
Hence, we have with $\mathbf{X}=(X_{0},\ldots,X_{N})$
\begin{multline*}
\ell_{c}\left(  \sigma,\lambda,\theta;\mathbf{X}\right)  =-Z_{0}\log
\sigma-\frac{1}{2\sigma^{2}\Delta t}\bigl[Z_{1}-2\lambda\theta\Delta
tZ_{2}-2(1-\lambda\Delta t)Z_{3}\\
+\lambda^{2}\theta^{2}\Delta t^{2}Z_{4}+2\lambda\theta\Delta t(1-\lambda\Delta
t)Z_{5}(1-\lambda\Delta t)^{2}Z_{6}\bigr].
\end{multline*}

Then we do the E-step. Given guesses $\sigma^{n},\lambda^{n},\theta^{n}$ for
the parameters, let
\begin{equation}
z_{i}:=\mathbb{E}_{\sigma^{n},\lambda^{n},\theta^{n}}\left[  \left.
Z_{i}\right\vert X_{i_{0}}=x_{i_{0}},\ldots,X_{i_{r}}=x_{i_{r}}\right]  ,\quad
i=1,\ldots,6,\label{infE}%
\end{equation}
and observe that
\begin{align*}
Q(\sigma,\lambda,\theta;\sigma^{n},\lambda^{n},\theta^{n};x_{i_{0}}%
,\ldots,x_{i_{r}}) &  :=\mathbb{E}_{\sigma^{n},\lambda^{n},\theta^{n}}\left[
\left.  \ell_{c}\left(  \sigma,\lambda,\theta;\mathbf{X}\right)  \right\vert
X_{i_{0}}=x_{i_{0}},\ldots,X_{i_{r}}=x_{i_{r}}\right]  \\
&  =-z_{0}\log\sigma-\frac{1}{2\sigma^{2}\Delta t}\bigl[z_{1}-2\lambda
\theta\Delta tz_{2}-2(1-\lambda\Delta t)z_{3}\\
&  \quad+\lambda^{2}\theta^{2}\Delta t^{2}z_{4}+2\lambda\theta\Delta
t(1-\lambda\Delta t)z_{5} + (1-\lambda\Delta t)^{2}z_{6}\bigr].
\end{align*}
Now the trouble is that for $\gamma\geq1/2$ some of the expectations in
(\ref{infE}), in particular $z_{4},$ may fail to exist. As such, this example
shows that in certain cases the expectation of the log-likelihood statistic in
the EM algorithm does not exist and that, as a consequence, the EM algorithm
can not be applied. We underline that existence of the log-likelihood
expectation is a premises for the EM algorithm in general and is not related
to the particular approach presented in this paper.

In the case $\gamma<1/2$ the expectations in (\ref{infE}) do exist and we may
proceed with first order conditions for finding the maximum of
\[
(\sigma,\lambda,\theta)\mapsto Q(\sigma,\lambda,\theta;\sigma^{n},\lambda
^{n},\theta^{n};x_{i_{0}},\ldots,x_{i_{r}}).
\]
We have that
\begin{align*}
\partial_{\sigma}Q &  =-\frac{z_{0}}{\sigma}+\frac{1}{\sigma^{3}\Delta
t}\bigl[z_{1}-2\lambda\theta\Delta tz_{2}-2(1-\lambda\Delta t)z_{3}\\
&  \quad+\lambda^{2}\theta^{2}\Delta t^{2}z_{4}+2\lambda\theta\Delta
t(1-\lambda\Delta t)z_{5}(1-\lambda\Delta t)^{2}z_{6}\bigr],\\
\partial_{\lambda}Q &  =-\frac{1}{2\sigma^{2}\Delta t}\bigl[-2\theta\Delta
tz_{2}+2\Delta tz_{3}+2\lambda\theta^{2}\Delta t^{2}z_{4}\\
&  \quad+2\theta\Delta t(1-2\lambda\Delta t)z_{5}-2\Delta t(1-\lambda\Delta
t)z_{6}\bigr],\\
\partial_{\theta}Q &  =-\frac{\lambda}{2\sigma^{2}\Delta t}\bigl[-2\Delta
tz_{2}+2\lambda\theta\Delta t^{2}z_{4}+2\Delta t(1-\lambda\Delta
t)z_{5}\bigr].
\end{align*}
and we so obtain the maximizers given by
\begin{align*}
\sigma^{2} &  =\frac{z_{3}^{2}z_{4}-2z_{2}z_{3}z_{5}+z_{1}z_{5}^{2}+z_{3}%
^{2}z_{6}-z_{1}z_{4}z_{6}}{\Delta tz_{0}\left(  z_{5}^{2}-z_{4}z_{6}\right)
},\\
\lambda &  =\frac{z_{3}z_{4}-z_{2}z_{5}+z_{5}^{2}-z_{4}z_{6}}{\Delta t\left(
z_{5}^{2}-z_{4}z_{6}\right)  },\\
\theta &  =\frac{z_{3}z_{5}-z_{2}z_{6}}{z_{3}z_{4}-z_{2}z_{5}+z_{5}^{2}%
-z_{4}z_{6}}.
\end{align*}
For the forward-reverse algorithm, we finally need to specify the reverse
chain. In this case, we propose to take the reverse chain
\begin{equation}
Y_{n+1}=Y_{n}-\lambda\left(  \theta-Y_{n}\right)  \Delta t+\sigma\left\vert
Y_{n}\right\vert ^{\gamma}\Delta\widetilde{W}_{n+1}.\label{eq:cir-reverse}%
\end{equation}
In order to get the dynamics of $\mathcal{Y}$, we need to derive the
normalization function $\psi$ between the one-step transition densities $p$ of
the forward and $q$ of the reverse processes. (We suppress the indices as we
are in a time-homogeneous situation.) For~(\ref{eq:cir}) together
with~(\ref{eq:cir-reverse}) the one-step transition densities are normal
densities in the forward variables,
\begin{gather*}
p(x,y)=\frac{1}{\sqrt{2\pi\Delta t}\sigma\left\vert x\right\vert ^{\gamma}%
}\exp\left(  -\frac{\left(  y-x-\lambda(\theta-x)\Delta t\right)  ^{2}%
}{2\sigma^{2}\left\vert x\right\vert ^{2\gamma}\Delta t}\right)  ,\\
q(y,z)=\frac{1}{\sqrt{2\pi\Delta t}\sigma\left\vert y\right\vert ^{\gamma}%
}\exp\left(  -\frac{\left(  z-y+\lambda(\theta-y)\Delta t\right)  ^{2}%
}{2\sigma^{2}\left\vert y\right\vert ^{2\gamma}\Delta t}\right)  .
\end{gather*}
Hence, we get
\begin{equation}
\psi(y,z)=\frac{p(z,y)}{q(y,z)}=\left\vert \frac{y}{z}\right\vert ^{\gamma
}\exp\left(  -\frac{1}{2\sigma^{2}\Delta t}\left[  \frac{\left(
y-z-\lambda(\theta-z)\Delta t\right)  ^{2}}{\left\vert z\right\vert ^{2\gamma
}}-\frac{\left(  z-y+\lambda(\theta-y)\Delta t\right)  ^{2}}{\left\vert
y\right\vert ^{2\gamma}}\right]  \right)  .\label{eq:psi-cir}%
\end{equation}

\bibliographystyle{plainnat}
\bibliography{BMSref}

\end{document}